\documentclass[11pt,oneside]{amsart}

\usepackage{piccoh-preamble}

\title{Invertibility and parity in symmetric monoidal categories}
\dedicatory{Dedicated to Bob Par\'e on the occasion of his 80th birthday.}
\authorGurski
\authorJohnson

\DTMsavedate{finaldate}{2026-03-31} 
\date{\DTMusedate{finaldate}} 

\hypersetup{bookmarksdepth=2}

\newcommand{\printkwds}{invertible object, coherence, super integers, flexible algebra, 2-monad, Picard category}

\hypersetup{ 
  pdfkeywords={\printkwds},
  pdfauthor={\authors},
}


\newcommand{\anb}{\{a,b\}}
\newcommand{\aisob}{\{a \iso b\}}
\newcommand{\atob}{\{a\, \begin{tikzpicture}[baseline={(0,-3.5pt)}]
    \draw[1cell,->] (0,0) to (.45*\arrowlen,0);
  \end{tikzpicture}\, b\}}
\newcommand{\atotob}{\{a\, \begin{tikzpicture}[baseline={(0,-3.5pt)}]
    \def\arrshift{2.0pt}
    \begin{scope}[shift={(0,\arrshift)}]
      \draw[1cell,->] (0,0) to (.45*\arrowlen,0);
    \end{scope}
    \begin{scope}[shift={(0,-\arrshift)}]
      \draw[1cell,->] (0,0) to (.45*\arrowlen,0);
    \end{scope}
  \end{tikzpicture}\, b\}}

\newcommand{\term}[1]{\emph{#1}} 

\newcommand{\ang}[1]{\langle #1 \rangle}

\newcommand{\Z}{\zZ}
\renewcommand{\P}{\zP} 
\renewcommand{\S}{\zS} 

\DeclareMathOperator{\p}{\mathsf{p}}
\DeclareMathOperator{\e}{\mathsf{e}}
\DeclareMathOperator{\einv}{\wt{\e}}

\newcommand{\fet}{\mathtt{8}}
\newcommand{\figC}{\mathtt{C}}
\newcommand{\figH}{\mathtt{H}}

\definecolor{cola}{HTML}{EC4300} 

\newcommand{\pinv}[1]{{#1}^{\mathrm{inv}}}
\newcommand{\textc}[1]{\color{cola}{\textsc{\small[#1]}}}
\newcommand{\ochi}{\wt{\chi}}

\newcommand{\qmarg}{(?)} 

\def\oleta{\ol{\eta}}
\def\olepz{\ol{\epz}}

\begin{document}

\begin{abstract}
  We introduce a notion of parity for formal morphisms between invertible objects and use it to prove a corresponding coherence theorem.
  Parity is conceptually similar to the sign of underlying permutations, but not defined as such.
  To give complete details, this work includes a thorough treatment of the free permutative category on an invertible generator, its skeletal model, known as the super integers, and an equivalence between them classified by the pair of integers $\pm$1.
  Our approach is organized and clarified as an application of 2-monadic algebra, particularly the concept of flexibility and the Lack model structure.
  The final section contains a number of examples applying the main results.
\end{abstract}

\maketitle
\tableofcontents 

\section{Introduction}\label{sec:intro}

This paper introduces \term{parity} to prove a coherence theorem for morphisms between invertible objects in symmetric monoidal categories.
As a starting point, Mac~Lane's Symmetric Coherence Theorem \cite[IV.4]{ML98Categories} shows that general morphisms in a free symmetric monoidal category are determined by their \term{underlying permutation}.
Invertibility of objects introduces additional relations---namely abelianization of automorphism groups---and thus one expects morphisms between invertible objects to be characterized by the \term{sign} of their underlying permutation.

This is the core idea for our definition of parity, but subtleties arise because invertibility means that the concept of underlying permutation---even up to sign---is difficult to define precisely.
Consider, for example, that the unit of a symmetric monoidal category with invertible objects generally has nontrivial automorphisms arising as the Euler characteristics of those invertible objects.
And yet, the unit is an empty sum; there are no odd permutations of a 0-element set.
Or simply observe that cancellation of invertible pairs means the number of summands in the source and target of a morphism need not be equal.

On this last point, it will be the case that the number of invertible summands is well-defined modulo two.
Moreover, signed counting of summands can associate a well-defined integer to every object in the free permutative category on a single invertible generator, such that every morphism has the same integer associated to its source and target.
Making these details clear and usable in calculations is one of the primary motivations for this work.

The main coherence result is stated as follows.
Here, and throughout the rest of this document, we restrict to \emph{permutative} categories: symmetric monoidal categories whose underlying monoidal structure is strict.
Mac~Lane's Symmetric Coherence Theorem \cite{ML98Categories} implies that each general symmetric monoidal category is equivalent, as such, to a permutative category.
Thus, for the purposes of coherence, there is no loss of generality and a nontrivial reduction in complexity.
\begin{thm}[Coherence via parity]\label{thm:free-inv-coherence}
  Let $G$ be a finite set and suppose that $s,t\cn z \to w$ is a pair of parallel morphisms in $\PP G$ \cref{equation:PPG-defn}, the free permutative category on a set of invertible generators $G$.
  Then $s$ and $t$ are equal as morphisms in $\PP G$ if and only if they have the same $a$-parity for each $a \in G$ \cref{eq:apar}.
\end{thm}

The proof of \cref{thm:free-inv-coherence} is given in \cref{sec:z-comp-2}, following \cref{rmk:betaab-triv}.
We now outline three main features of our treatment.

\subsection*{2-monadic framework}
Firstly, we use a 2-monadic approach to explain two different constructions of the free permutative category on a single generator.
The first of these, $\P$, is defined abstractly via a sequence of pushouts, and therefore has a corresponding universal property by construction.
However, the abstract construction of $\P$ makes it unclear how to characterize morphisms for a coherence statement.

The second construction, $\Z$, is defined directly by a list of objects (integers), morphisms (integers mod 2), and axioms.
The concrete definition makes it trivial to characterize morphisms, and this is how our parity invariant is defined.
However, the skeletal definition of $\Z$ makes it is difficult to show that $\Z$ satisfies the expected universal property for being freely generated by an invertible object.

We construct an equivalence
$K \cn \P \to \Z$ in \cref{thm:ZZ-equiv-P1}, using the 2-monadic concept of \term{flexibility} for $\P$ and a number of detailed calculations for abstract invertible objects.
This equivalence $K$ gives a direct and easily-computable definition of parity for morphisms in $\P$, the free permutative category on a single invertible object.

\subsection*{Reduction via product projection}
Secondly, to move beyond the single-generator case, we recall the equivalence between 2-categorical products and coproducts of permutative categories as \cref{thm:smbperm1427}.
Such a result shows that permutative categories have bicategorical direct sums, defined by either products or coproducts.
Since a free construction on multiple generators will necessarily have the universal property of a coproduct, we use the equivalence with products to give a componentwise reduction of coherence.

Such a strategy for permutative categories (i.e., without assuming invertibility) was already outlined in \cite[Section~12]{GJalmorcoh}.
Consider a pair of parallel morphisms, $s$ and $t$, in $\SS G$, the free permutative category generated by a set $G$ (\cref{thm:permcat-monadic,rmk:free-permcats}).
Instead of comparing the underlying permutations of $s$ and $t$ as in Mac~Lane's Symmetric Coherence \cref{thm:permcat-coherence-one-obj}, we explain in \cref{thm:diagrcoh-s,example:permcoh-comptise} that one only needs to consider the \emph{self}-permutations of $s$ and $t$ with respect to each of the elements of $G$.
The self-permutation of a morphism with respect to a generator $a \in G$ is called its $a$-permutation, and
$s=t$ (``the diagram commutes'') if and only if $s$ and $t$ have the same $a$-permutations for each $a \in G$.
This reduction is an interesting, but not necessarily essential, simplification of \cref{thm:permcat-coherence-one-obj} that comes as a direct consequence of the direct sums: the $a$-permutation of a morphism is computed by projection onto the factor corresponding to $a$, and such an operation detects equality because coproducts are equivalent to products.

In the case of invertible objects, the equivalence of coproducts and products is essential to our multi-generator results.
We show that it suffices to consider \emph{self}-parity for each object in the source and target of a morphism between products of invertible objects.
Here, the notion of self-parity necessarily includes an object \emph{and} its inverse---a subtlety that is made precise by the equivalence of coproducts with products.
This reduction to singly-generated factors is how our methods avoid the difficulty of defining underlying permutations for morphisms between invertible objects.

\subsection*{Examples}
The third feature of our treatment is its examples.
Our construction of the equivalence $K\cn \P \to \Z$ depends on a number of calculations whose details demonstrate a range of useful techniques.
While several of these are standard in the literature and well known to experts, others are less so.

Following the proof of \cref{thm:free-inv-coherence}, we give several example applications.
These include a more involved example of an additivity property that is known in the broader trace literature, but whose proof in the special case of invertible objects is somewhat simpler and demonstrates how our coherence theorem would be used for additional work with invertible objects.
Finally, we discuss conjugation of an arbitrary object in a permutative category by an invertible one, using parity at multiple points to simplify the computations.

\subsection*{Outline}

\cref{sec:perm-cats,sec:tmp-coh} review the basic definitions and results for permutative categories.
This review includes: the equivalence between coproducts and products in \cref{thm:smbperm1427}; Mac~Lane's Symmetric Coherence \cref{thm:permcat-coherence-one-obj}; and the componentwise reduction in \cref{thm:diagrcoh-s}.

\cref{sec:inv-objs,sec:calcs} review the basic definitions and properties of invertible objects.
\cref{sec:flex-mod} reviews the essential features of flexibility and defines the representing permutative category, $\P$.
\cref{sec:superZ} then introduces the super integers, $\Z$, and proves the equivalence $K \cn \P \to \Z$ in \cref{thm:ZZ-equiv-P1}.

The final \cref{sec:z-comp-2} contains the definition of parity and the application of \cref{thm:smbperm1427,thm:ZZ-equiv-P1} to our main coherence result, \cref{thm:free-inv-coherence}.
The proof is followed by numerous examples illustrating its use.

\subsection*{Relation to literature}

\subsubsection*{Coherence}
Coherence for various monoidal structures is an area of significant interest.
In addition to Mac~Lane's work \cite{MLan63Natural,ML98Categories} on the plain and symmetric monoidal cases, there is related work of Kelly-Laplaza \cite{KL1980Coherence} for general compact closed categories.
Our review of symmetric coherence originates with work of May in \cite{may72geo,May1974Einfty}, with revisions to identify the connections with 2-dimensional algebra.

Work of Baez-Lauda \cite{BL2004Higher} uses prior work of Ulbrich \cite{Ulb1984Kohaerenz} and Laplaza \cite{Lap1983Coherence} to give coherence and strictification for (non-symmetric) categorical groups.
The resulting coherence parallels that of monoidal categories: every formal diagram commutes.

Work of Dugger \cite{Dug2014Coherence} specializes the compact closed coherence of \cite{KL1980Coherence} to give a coherence theorem that is essentially equivalent to our \cref{thm:free-inv-coherence}.
Dugger's notion of \term{self-twists} is comparable to our $a$-parity for invertible pairs $(a,a')$.
Our methods have three essential differences from those of Dugger:
\begin{enumerate}
\item Our work does not depend on the more general results of Kelly-Laplaza \cite{KL1980Coherence}.
\item Our componentwise projection in \cref{thm:diagrcoh-s} simplifies the presentation somewhat.
\item Our construction of $\P$, the free permutative category on an invertible generator, and related analysis of the equivalence $K \cn \P \to \Z$, clarifies the relation between freeness, strictness, and coherence in this context.
\end{enumerate}

\subsubsection*{2-dimensional algebra and flexibility}

Our approach to 2-dimensional algebra follows the 2-monadic framework of Blackwell-Kelly-Power \cite{BKP1989Two}, later expanded by Lack \cite{Lac02Codescent,Lack2007Homotopy}.
Note, however, that the essential ideas of flexibility have origins in Kelly's work on doctrinal adjunction \cite{Kelly1974Doctrinal,Kel1974Coherence}.

We make crucial use of the Quillen model structure on algebras over a 2-monad developed by Lack in \cite{Lack2007Homotopy}.
This model structure is transferred from the canonical model structure on $\cat$ (also known as the trivial or folk model structure) from Joyal-Tierney \cite{JT1991Strong}.
In Lack's model structure, and therefore in the context of permutative categories, the flexible algebras are precisely the cofibrant ones (\cref{prop:cofibrant-flexible}).
This is how we show that $\P$ is flexible (\cref{prop:pflex})---an essential part of the proof of \cref{thm:ZZ-equiv-P1}.

\subsubsection*{The super integers}

The permutative category $\Z$ has been studied under various names, with motivations connected to a range of mathematical disciplines.
Our construction of $\Z$ follows that of \cite[Section~3]{JO2012Modeling} and \cite[Example~3.1.7 (e)]{GK2014Symmetric}.
The name \emph{super integers} for $\Z$ first appeared in \cite[Example~2.3.10]{Hor2020Cohomology}.

An equivalence $\P \hty \Z$ is part of the folklore in this area, being understood implicitly or explicitly in these and other references.
Explanations with varying levels of detail appear, for example, in
\cite[Proposition~3.1]{JO2012Modeling}, 
\cite[Section~1.22]{Dug2014Coherence}, and
\cite[Proposition~3.1.2]{Kap2021Supergeometry}. 
Although $\Z$ is often constructed in detail, the authors are not aware of such a careful treatment for $\P$ or a 2-monadic approach to the equivalence $K$ in \cref{thm:ZZ-equiv-P1}.

\section{Permutative categories}
\label{sec:perm-cats}

This section provides the fundamental background on permutative categories.
Here we establish our notation and terminology for operations within a given permutative category, and for the 2-category of which they are the objects.
We also remind readers, in \cref{thm:smbperm1427}, that the 2-category of permutative categories admits bicategorical direct sums: finite coproducts and finite products coincide up to equivalence.

\begin{defn}\label{defn:permutative}
  A \term{permutative category} $A$ consists of a strict monoidal category $(A, +, 0)$ together with a natural isomorphism called the \term{symmetry},
  \[
    \begin{tikzpicture}[x=18mm,y=10mm]
      \draw[0cell] 
      (0,0) node (xy) {A \times A}
      (1,-1.25) node (yx) {A \times A}
      (2,0) node (xyagain) {A}
      ;
      \draw[1cell] 
      (xy) edge node {+} (xyagain)
      (xy) edge[swap] node {\tau} (yx)
      (yx) edge[swap] node {+} (xyagain)
      ;
      \draw[2cell]
      (1,-.5) node[rotate=0] {\be \Downarrow};
    \end{tikzpicture}
  \]
  where $\tau \cn A \times A \to A \times A$ is the symmetry
  isomorphism in $\Cat$, such that the following axioms hold for all
  objects $x,y,z$ of $A$:
  \begin{equation}\label{eq:symm-axioms}
    \begin{aligned}
      \beta_{y,x} \beta_{x,y} & = \id_{x + y} \andspace\\
      \beta_{x, y + z} & = (\id_y + \beta_{x,z}) \circ (\beta_{x,y} + \id_z).
    \end{aligned}
  \end{equation}
  Note that these imply $\beta_{0,x} = \id_{x} = \beta_{x,0}$ for each object $x$ in $A$.
\end{defn}

\begin{rmk}[Notation for monoidal structures]\label{rmk:concat}
  We will usually use $+$ for the monoidal operation and $0$ for the unit in a permutative category, and we will always refer to this operation as the \term{monoidal sum}.
  At times, particularly in larger diagrams, we will omit these symbols and use concatenation as the operation.
  For a morphism $f \cn x \to y$, we will often write $a + f$ for the morphism $\id_a + f \cn a + x \to a + y$; when the $+$ symbol is being suppressed, this morphism might instead be written $af$ or $1f$.
\end{rmk}

\begin{notn}[Multiplication by natural numbers]\label{notn:n-dot}
  Let $(A, +, 0)$ be a permutative category with object $x \in A$.
  For each natural number $k \ge 0$, we will write
  \begin{equation}\label{eq:n-dot}
    k \cdot x = \overbrace{x + x + \cdots + x}^{\text{$k$ copies}}
  \end{equation}
  for the $k$-fold iterated sum of $x$.
  Note that $0 \cdot x = 0$ because the empty sum is the monoidal unit.
  Similarly, for morphisms $f$ in $A$, we let $k \cdot f$ denote the $k$-fold iterated sum of $f$.
\end{notn}

\begin{defn}\label{defn:smfunctor}
 Let $(A,+,0^A)$ and $(B,+,0^B)$ be permutative categories.
 A \term{symmetric monoidal functor} between them consists of a functor $F\colon A \to B$, a natural isomorphism $F_2$ called the \term{monoidal constraint} with components
 \[
 F_{2; x,y}\colon Fx + Fy \to F(x + y), \forspace x,y \in A,
 \]
 and an isomorphism
 \[F_0\colon 0^B \to F0^A,\]
 called the \term{unit constraint}.
 These data satisfy the following associativity, unity, and symmetry axioms for all objects $x,y,z$ in $A$:
 \begin{equation}\label{eq:smfunctor}
   \begin{aligned}
     F_{2;x,y+z} \, \circ \,  \bigl( \id_{Fx} +F_{2;y,z}\bigr)
     & = F_{2;x+y,z} \, \circ \,  \bigl( F_{2;x,y} + \id_{Fz}\bigr),\\
     F_{2;0,x} \,\circ\, \bigl( F_0 + \id_{Fx} \bigr)
     & = \id_{Fx}, \\
     F_{2;x,0} \,\circ\, \bigl( \id_{Fx} + F_0 \bigr)
     & = \id_{Fx}, \andspace \\
     F_{2;y,x} \,\circ\, \beta_{Fx,Fy}
     & = F\bigl(\beta_{x,y}\bigr) \,\circ\, F_{2;x,y}.
   \end{aligned}
 \end{equation}
 Note that in the presence of symmetries, either of the unit axioms implies the other.

 In the literature, what we have called a symmetric monoidal functor is sometimes also called a \term{strong symmetric monoidal functor}, indicating that the monoidal and unit constraints are isomorphisms.
 If $F_2$ and $F_0$ are identities, we say $F$ is a \term{strict symmetric monoidal functor}.
\end{defn}

\begin{defn}\label{defn:montransf}
  Let $F,G \cn A \to B$ be symmetric monoidal functors between permutative categories.
  A \term{monoidal transformation} $\phi \cn F \To G$ consists of a natural transformation $\phi$ such that the following monoidal naturality and unit axioms hold for each pair of objects $x,y \in A$:
  \begin{equation}\label{eq:montransf}
    \begin{aligned}
      \phi_{x + y} \circ F_{2;x,y}
      & = G_{2;x,y} \circ \bigl( \phi_x + \phi_y \bigr) \andspace \\
      \phi_{0} \circ F_0
      & = G_0.
    \end{aligned}
  \end{equation}
  Composition and identities of monoidal transformations are determined componentwise.
  A monoidal transformation between strict symmetric monoidal functors is defined in the same way.
\end{defn}

\begin{notn}\label{notn:permcat}
  We write $\permcat$ for the 2-category of permutative categories, symmetric monoidal functors, and monoidal transformations.
  We write $\permcats$ for the sub 2-category whose 1-cells are strict symmetric monoidal functors, and
  \begin{equation}\label{eq:incl-j}
    j \cn \permcats \to \permcat
  \end{equation}
  for the inclusion 2-functor that is the identity on 0-, 1-, and 2-cells.
\end{notn}

\begin{rmk}\label{rmk:permcat-complete-cocomplete}
  Each of the 2-categories $\permcat$ and $\permcats$ is complete and cocomplete in the 2-categorical ($\Cat$-enriched) sense. 
  These facts follow abstractly from (2-)monadicity results; see \cite{BKP1989Two} or \cite{Kelly1972Many,Kelly1972Coherence}.

  For most of this article, our use of co/limits will require nothing deeper than their existence and universal properties.
  The sole exception is \cref{thm:smbperm1427}, immediately following this remark.

  For the 2-dimensional universal properties of colimits in particular, we refer the reader to \cite[Explanation~5.3.6]{JY212Dim} for more detail.
  Other standard references, beyond those noted above, include \cite{Kel89Elementary} and \cite{Lac10Companion}. 
\end{rmk}

\begin{thm}[{\cite[Theorem~14.27]{GJOsmbperm}}]\label{thm:smbperm1427}
  Suppose given permutative categories $A_i$ for $i \in \{1,\ldots,n\}$.
  There is a strict symmetric monoidal functor $I$
  \begin{equation}\label{eq:I}
    \coprod_{i = 1}^n A_i \fto{I} \prod_{i = 1}^n A_i
  \end{equation}
  such that the following statements hold.
  \begin{enumerate}
  \item Each composite of $I$ with the coproduct inclusions and product projections,
    \[
      A_i \to
      \coprod_{i = 1}^n A_i \fto{I} \prod_{i = 1}^n A_i
      \to A_j,
    \]
    is the identity on $A_i$ if $i = j$ and constant at the monoidal unit of $A_j$ otherwise.
  \item $I$ is an equivalence of permutative categories.
  \end{enumerate} 
\end{thm}
\begin{rmk}\label{rmk:direct-sums}
  \cref{thm:smbperm1427} shows that $\permcats$ has a bicategorical analogue of direct sums, categorifying the corresponding fact for abelian groups.
  We refer the reader to \cite[Section~14]{GJOsmbperm} for further discussion and many more details of this phenomenon.
\end{rmk}

\section{Coherence for permutative categories}
\label{sec:tmp-coh}

Now we turn to coherence for permutative categories.
Here, we make use of \cref{thm:smbperm1427} to give a componentwise variant of the standard symmetric monoidal coherence due to Mac~Lane in \cite[XI.1]{ML98Categories}.
In addition to providing somewhat simpler criteria for checking underlying permutations, this development parallels our componentwise approach for invertible coherence results in \cref{sec:z-comp-2}.

\begin{notn}[Symmetric groups]\label{notn:symm-grps}
  Suppose $n \ge 0$ is a natural number.
  We write $\Sigma_n$ for the symmetric group on $n$ letters.
\end{notn}

\begin{notn}[Translation categories]\label{notn:translation}
  Suppose $S$ is a group.
  We write $ES$ for the \term{translation category} that has object set $S$ and a unique isomorphism $s \cong t$ for each pair $s, t \in S$.
  This construction is also known as the \term{translation groupoid} in the theory of group actions on sets \cite[Example~1.5.13(v)]{Rie2017CTC}.
\end{notn}

The essential ideas of the following result, \cref{thm:permcat-monadic}, were known to May \cite{may72geo,May1974Einfty}, although that work does not take a 2-monadic perspective.
\begin{thm}[{\cite[Theorem~15.12]{CG25Operads}}]\label{thm:permcat-monadic}
  The categories $E \Sigma_n$ assemble to give a symmetric operad, denoted $E\Sigma$, in $\cat$. The symmetric operad $E \Sigma$ induces a 2-monad $\SS$ on $\cat$ for which $\SAlgs \cong \permcats$ as 2-categories over $\cat$.
\end{thm}
\begin{explanation}\label{explanation:permcat-monadic}
  The work of \cite{CG25Operads}, following earlier work of \cite{Kelly1972Coherence}, develops a theory of presentations for 2-monads (there called clubs) of the form $E\Lambda$, where $\Lambda$ is an \emph{action operad} \cite[Definition~4.1]{CG25Operads}.
  \cref{thm:permcat-monadic} follows from a special case of \cite[Theorem~15.12]{CG25Operads} with $\Lambda = \Sigma$, the action operad for the symmetric groups.
  As explained in \cite[Example~15.13(2)]{CG25Operads}, the data and axioms defining permutative categories correspond precisely to a presentation for $\Sigma$ as an action operad.
  This correspondence establishes the isomorphism of 2-categories asserted in \cref{thm:permcat-monadic}.
  See \cite[Examples~7.6 and~15.13(2)]{CG25Operads} for further explanation.
\end{explanation}

\begin{notn}[Free permutative categories]\label{notn:braces}
  If $C$ is a small category, we write $\SS C$ for the free permutative category generated by $C$.
  For each permutative category $A$, there is thus a free-forgetful 2-adjunction
  \begin{equation}\label{eq:SS-free-forget}
    \permcats(\SS C, A) \iso \cat(C,A).
  \end{equation}

  We will often apply the 2-monad $\SS$ to finite sets, treated as discrete categories, and write $\SS G$ for the free permutative category generated by a finite set $G$.
  We will use the special case $n = 1$ frequently, and write $\S$ or $\S\{x\}$ for $\SS\{x\}$.
  For general finite sets $X$, observe that $\SS$ is a left 2-adjoint and therefore the free permutative category $\SS X$ can be expressed as a coproduct in $\permcats$:
  \begin{equation}\label{eq:SSG-coprodS}
    \SS X \cong \coprod_{x \in X} \S\{x\}.
  \end{equation} 
\end{notn}

\begin{rmk}\label{rmk:free-permcats}
  The free permutative category $\SS C$ for a small category $C$ can be described as follows.
  From the operadic description of $\SS$ in \cref{thm:permcat-monadic}, we have
  \begin{equation}\label{eq:SS-C}
    \SS C = \coprod_{n \in \mathbbm{N}} E \Sigma_n \times_{\Sigma_n} C^n.
  \end{equation}
  The category $E \Sigma_n$ has a right action of the group $\Sigma_n$ by multiplication, while the category $C^n$ has a left action of the group $\Sigma_n$ by the formula
  \[
    \sigma \cdot (c_1, \ldots, c_n) = (c_{\sigma^{-1}(1)}, \ldots, c_{\sigma^{-1}(n)})
  \]
  on objects $\ang{c} = (c_1,\ldots,c_n)$, and similarly for morphisms. 
  The category $E \Sigma_n \times_{\Sigma_n} C^n$ is then given as the coequalizer
  \[
    E\Sigma_n \times \Sigma_n \times C^n \rightrightarrows 
    E\Sigma_n \times C^n \to E \Sigma_n \times_{\Sigma_n} C^n
  \]
  of these right and left actions.
  Every object of $\SS C$ is thus an equivalence class of pairs $\bigl( \sigma, (c_1, \ldots, c_n) \bigr)$, where the $c_i$ are objects of $C$ and $\sigma \in \Sigma_n$.
  The equivalence relation is given by
  \[
    \bigl( \sigma\tau, (c_1, \ldots, c_n) \bigr) \sim \bigl( \sigma, (c_{\tau^{-1}(1)}, \ldots, c_{\tau^{-1}(1)}) \bigr).
  \]
  Writing equivalence classes with square brackets, we have the equality
  \[
    [\sigma, (c_1, \ldots, c_n)] = [\id, (c_{\sigma^{-1}(1)}, \ldots, c_{\sigma^{-1}(n)})].
  \]
  Therefore, every object of $\SS C$ is uniquely expressed as a pair $(n, \ang{c})$ where $n$ is a natural number and $\ang{c}$ is an $n$-tuple of objects of $C$.

  The same analysis applies to the morphisms of $\SS C$, where it is important to recall that there is a unique isomorphism $\sigma \cong \tau$ in $E \Sigma_n$ for any pair $\sigma, \tau \in \Sigma_n$.
  For any $n$-tuple $c_1, \ldots, c_n$ of objects of $C$ and any $\sigma \in \Sigma_n$, there is an isomorphism denoted $[\sigma, 1]$:
  \begin{equation}\label{eq:sigma1}
    [\id, (c_1, \ldots, c_n)] \fto[\cong]{[\sigma,1]} [\sigma, (c_1, \ldots, c_n)] = [\id, (c_{\sigma^{-1}(1)}, \ldots, c_{\sigma^{-1}(n)})].
  \end{equation}
  Taking care with the group actions, one verifies
  \[
    [\tau, 1] \circ [\sigma, 1] = [\tau \sigma, 1],
  \]
  where $\tau \sigma$ is the product as computed in $\Sigma_n$.
  In the discrete case, the category $C$ has no non-identity morphisms, so every morphism in $\SS C$ is of the form $[\sigma, 1]$ for some permutation $\sigma$. 
\end{rmk}

\begin{notn}\label{notn:underlying-perm-map}
  For any small category $C$, let $! \cn C \to \{x\}$ denote the unique functor to a terminal category and recall $\S = \SS\{x\}$ from \cref{notn:braces}.
  We write $\omega \cn \SS C \to \S$ for $\SS !$.
\end{notn}

\begin{defn}\label{defn:underlying-single}
  An automorphism $f \cn n \cdot x \to n \cdot x$ in $\S$ has \term{underlying permutation} $\sigma$ if $f = [\sigma, 1]$ as in \cref{eq:sigma1}.
  We write $\ups(f)$ for the underlying permutation of $f$, so $\ups(f) = \sigma$ if $f = [\sigma, 1]$.
  
  Suppose $G$ is a set, and $f \cn y \to z$ is a morphism in $\SS G$.
  We say that $f$ has \term{underlying permutation} $\sigma$ if $\omega(f) = [\sigma, 1]$, where $\omega$ is the functor from \cref{notn:underlying-perm-map}.
  As before, we write $\ups(f)$ for the underlying permutation, so $\ups(f) = \sigma$ if $\omega(f) = [\sigma, 1]$.
\end{defn}

The explanation in \cref{rmk:free-permcats} yields the following coherence theorem for permutative categories \cite[XI.1]{ML98Categories}.
Recall from \cref{notn:braces} that $\S = \SS\{x\}$ denotes the free permutative category on a generator $x$.
\begin{thm}\label{thm:permcat-coherence-one-obj}
  For each natural number $n$, the function 
  \begin{equation}\label{eq:Snn-Sigman}
    \begin{aligned}
      \Sigma_n & \to \S (n \cdot x, n \cdot x), \\
      \sigma & \mapsto [\sigma,1]
    \end{aligned}
  \end{equation}
  is an isomorphism of monoids.

  Suppose, furthermore, that $G$ is a finite set.
  Then parallel morphisms $s$ and $t$ are equal in $\SS G$ if and only if $\ups(s)=\ups(t)$.
\end{thm}

Recall from \cref{thm:smbperm1427} the equivalence $I$ between finite coproducts and products of permutative categories. 
\begin{defn}\label{defn:ISS}
  Suppose $G$ is a finite set.
  Then, for each $a \in G$, let $I_a$ denote the following composite of $I$ \cref{eq:I} with \cref{eq:SSG-coprodS} and the product projection.
  \begin{equation}\label{eq:ISS}
    \begin{tikzpicture}[x=12ex,y=8ex,baseline={(a.base)}]
      \draw[0cell,] 
      (0,0) node (a0) {\SS G}
      (a0)++(8ex,0) node[text depth=3pt] (a) {\coprod_{g \in G} \S\{g\}}
      (a)++(1.3,0) node[text depth=3pt] (b) {\prod_{g \in G} \S\{g\}}
      (b)++(1,0) node (c) {\S\{a\}}
      ;
      \draw[0cell] (a0)++(3.1ex,0) node {\cong};
      \draw[1cell] 
      (a) edge node {I} node['] {\hty} (b)
      (b) edge node {} (c)
      ;
      \draw[1cell,rounded corners]
      (a0) -- ++(0,.5) -- node {I_a} ($(c)+(0,.5)$) -- (c)
      ;
    \end{tikzpicture}\vspace{.5pc}
  \end{equation}
  The \term{$a$-permutation} of a morphism $f$ in $\SS G$, denoted $\ups_a(f)$, is defined to be the underlying permutation of the morphism $I_a f$ in $\S\{a\}$:
  \begin{equation}\label{eq:a-perm}
    \ups_a(f) = \ups(I_a f).
  \end{equation}
\end{defn}

Now, by \cref{thm:smbperm1427,defn:ISS},
we reformulate \cref{thm:permcat-coherence-one-obj} using $a$-permutations instead of the underlying permutation.
\begin{thm}[{\cite[Theorem~12.7]{GJalmorcoh}}]\label{thm:diagrcoh-s}
  Let $G$ be a finite set and suppose that $s,t\cn z \to w$ is a pair of parallel morphisms in the free permutative category $\SS G$.
  If $s$ and $t$ have the same $a$-permutation for each $a \in G$, then they are equal as morphisms in $\SS G$.
\end{thm}

\begin{rmk}\label{rmk:coherence-parallel}
  Note, in the statements of \cref{thm:permcat-coherence-one-obj,thm:diagrcoh-s}, the condition that $s$ and $t$ are \emph{parallel} in $\SS G$ is essential.
  If $G$ has more than one element, there will be \emph{non}-equal morphisms in $\SS G$ that are \emph{not} parallel, but do have the same underlying permutation via $\omega\cn \SS G \to \S$ and also the same $a$-permutation via $I_a$ for each $a$ in $G$.
  In particular, there will be non-identity morphisms $f$, with non-equal source and target, such that the $a$-permutation $\ups_a(f)$ is an identity for each $a \in G$.
  \cref{example:permcoh-comptise} includes several such morphisms.
\end{rmk}

\begin{example}\label{example:permcoh-comptise}
  Suppose that $a$ and $b$ are objects of a permutative category $A$, and consider whether the following diagram commutes.
  \begin{equation}\label{eq:mult2-symm}
    \begin{tikzpicture}[x=30ex,y=9ex,vcenter]
      \draw[0cell] 
      (0,0) node (a) {a+a+b+b}
      (a)++(1,0) node (b) {b+b+a+a}
      (a)++(0,-1) node (c) {a+b+a+b}
      (c)++(1,0) node (d) {b+a+b+a}
      ;
      \draw[1cell] 
      (a) edge node {\beta_{a+a,b+b}}
      (b)
      
      (b) edge node {\id_a + \beta_{a,b} + \id_b}
      (d)
      
      (a) edge['] node {\id_a + \beta_{a,b} + \id_b}
      (c)
      
      (c) edge node {\beta_{a,b} + \beta_{a,b}}
      (d)
      ;
    \end{tikzpicture}
  \end{equation}
  By \cref{thm:diagrcoh-s}, it suffices to consider the $a$-permutation and $b$-permutation of each composite, rather than the total permutation of four summands.
  Each $a$- and $b$-permutation above is an identity, because none of the symmetry isomorphisms in \cref{eq:mult2-symm} permute $a$ with itself or $b$ with itself.
  So, \cref{thm:diagrcoh-s} gives a simple proof that \cref{eq:mult2-symm} commutes.
\end{example}
For additional examples and further discussion, see \cite[Sections~12 and~16]{GJalmorcoh}.

\section{Invertible objects}
\label{sec:inv-objs}

This section introduces the main object of study for this paper, invertible pairs.
These invertible pairs are structured algebraic versions of invertible objects, where invertibility for an object of a monoidal category is understood to mean with respect to the monoidal operation and only up to isomorphism.
As in many other higher categorical contexts, the apparent encumbrance of additional data for invertible pairs provides the boon of simple-to-write explicit formulas for our constructions later.
We will define a category, $\pinv{A}$, of invertible pairs in a permutative category $A$ in \cref{defn:Ainv}, and verify that this construction is 2-functorial with respect to all symmetric monoidal functors in \cref{prop:pinv-2fun}.

Several of the intermediate results in this section and the next are known to experts and available in various forms in the literature.
An incomplete list includes
\cite{KL1980Coherence,BL2004Higher,Abramsky2009No,Dug2014Coherence,HV2019Categories,HZ2023Duality}.
We include statements and proofs for clarity and for the reader's convenience, making no claim to originality.

\begin{defn}\label{defn:invertible}
  An object $x$ in a monoidal category $(A,\otimes, I)$ is said to be \term{invertible} if there is an object $x' \in A$ together with isomorphisms
  \begin{equation}\label{eq:invertible-isos}
    x \otimes x' \iso I \andspace
    x' \otimes x \iso I.
  \end{equation}
  An object $x'$ as above will be called a \term{weak inverse} to $x$.
\end{defn}

\begin{rmk}\label{rmk:triangle-and-8}
  If $A$ is a symmetric monoidal category then, in the context of \cref{defn:invertible}, the existence of either isomorphism in \cref{eq:invertible-isos} implies the existence of the other by symmetry.
  However, a pair of isomorphisms constructed in this way will generally \emph{not} satisfy the triangle identities \cref{eq:aa'-triang} required of an \term{invertible pair} (\cref{defn:inv-obj}).
  We explain this further from the perspective of coherence for invertible pairs in \cref{example:weak-vs-strict-inv}.
\end{rmk}

\begin{defn}\label{defn:inv-obj}
  Let $(A, +, 0, \beta)$ be a permutative category. An \term{invertible pair} in $A$ is a tuple
  \[
    \ul{a} = (a, a', \eta, \epz)
  \]
  consisting of
  \begin{itemize}
  \item objects $a$ and $a'$ in $A$, and
  \item isomorphisms $\eta \cn 0 \cong a'+a$ and $\epz \cn a+a' \cong 0$ in $A$.
  \end{itemize}
  The data of an invertible pair are required to make the following two diagrams commute.
  These diagrams are called the \term{triangle identities for an invertible pair}.
  \begin{equation}\label{eq:aa'-triang}
    \begin{tikzpicture}[x=18ex,y=8ex,vcenter]
      \draw[0cell] 
      (0,0) node (a) {a}
      (1,0) node (b) {a+a'+a}
      (1,-1) node (c) {a}     
      (2,0) node (d) {a'}
      (3,0) node (e) {a'+a+a'}
      (3,-1) node (f) {a'}
      ;
      \draw[1cell] 
      (a) edge node {a+\eta} (b)
      (b) edge node {\epz+a} (c)
      (a) edge[swap] node {\id_a} (c)
      (d) edge node {\eta+a'} (e)
      (e) edge node {a'+\epz} (f)
      (d) edge[swap] node {\id_{a'}} (f)
      ;
    \end{tikzpicture}
  \end{equation}
  The first coordinate $a$ for an invertible pair $\ul{a} = (a, a', \eta, \epz)$ is called the \term{base object} of $\ul{a}$.
\end{defn}

\begin{rmk}\label{rmk:inverse-equivs}
  Let $(A, +, 0, \beta)$ be a permutative category with objects $x$ and $a$.
  In the context of \cref{defn:invertible}, the isomorphisms \cref{eq:invertible-isos} imply that $x'$ is a weak inverse for $x$ if and only if the left translation functors
  \[
    x+\qmarg \cn A \lradj A \bacn x'+\qmarg
  \]
  form an equivalence of categories.

  In the context of \cref{defn:inv-obj}, the triangle identities \cref{eq:aa'-triang} imply that a quadruple $(a, a', \eta, \epz)$ is an invertible pair if and only if the left translation functors
  \[
    a+\qmarg\cn A \lradj A\bacn a'+\qmarg
  \]
  are a pair of \emph{adjoint} equivalences with unit and counit given by $\eta+\qmarg$ and $\epz+\qmarg$, respectively.
\end{rmk}

\begin{lem}\label{lem:weak-vs-structured-inv}
  Let $(A, +, 0, \beta)$ be a permutative category.
  An object $x$ is invertible in $A$ if and only if it is the base object of an invertible pair $\ul{x}$.
\end{lem}
\begin{proof}
  Recall that a pair of functors $(F,G)$ form an equivalence of categories if and only if they are adjoint equivalences (see \cite[IV.4]{ML98Categories} or \cite[Theorem~1.4]{Gur2012Biequivalences}).
  The result then follows from \cref{rmk:inverse-equivs} by considering $F=(x+\qmarg)$ and $G=(x'+\qmarg)$.
\end{proof}

\begin{rmk}\label{rmk:not-smf}
  We point out that $a+\qmarg$ is generally not a monoidal functor, much less a symmetric monoidal functor.
  However, functoriality of $+$ implies, for any $z \in A$ and each $x \in A$, that 
  \begin{equation}\label{eq:zqmark-hom}
    z+\qmarg\cn A(x,x) \to A(z+x,z+x)
  \end{equation}
  is a homomorphism of monoids with respect to composition.
  If $z$ is an invertible object, then $z+\qmarg\cn A \to A$ is an equivalence (\cref{rmk:inverse-equivs}) and hence \cref{eq:zqmark-hom} is an isomorphism of monoids.
\end{rmk}

\begin{defn}\label{defn:inv-obj-map}
  Let $(A, +, 0, \beta)$ be a permutative category, and let $(a, a', \eta, \epz)$ and $(b,b', \delta, \gamma)$ be invertible pairs in $A$.
  A \term{map of invertible pairs} from $(a, a', \eta, \epz)$ to $(b,b', \delta, \gamma)$ is a pair $(f, f')$ consisting of morphisms $f:a \to b$ and $f':a' \to b'$ in $A$ such that the following diagrams commute.
  \begin{equation}\label{eq:inv-obj-map}
    \begin{tikzpicture}[x=18ex,y=8ex,vcenter]
      \draw[0cell] 
      (0,0) node (a) {0}
      (1,0) node (b) {a'+a}
      (1,-1) node (c) {b'+b}
      
      (3,0) node (f) {0}
      (2,0) node (d) {a+a'}
      (2,-1) node (e) {b+b'}
      ;
      \draw[1cell] 
      (a) edge node {\eta} (b)
      (b) edge node {f'+f} (c)
      (a) edge[swap] node {\delta} (c)
      
      (d) edge['] node {f+f'} (e)
      (e) edge['] node {\gamma} (f)
      (d) edge node {\epz} (f)
      ;
    \end{tikzpicture}
    \vspace{-1pc} 
  \end{equation}
  \ 
\end{defn}

\begin{defn}\label{defn:Ainv}
Let $(A, +, 0, \beta)$ be a permutative category. The \term{category of invertible pairs} in $A$, denoted $\pinv{A}$, consists of
\begin{itemize}
\item objects the invertible pairs in $A$,
\item morphisms the maps of invertible pairs,
\item identities given by $(\id_a, \id_{a'})\cn (a, a', \eta, \epz) \to (a, a', \eta, \epz)$, and 
\item composition given by $(g, g') \circ (f, f') = (g \circ f, g' \circ f')$.
\end{itemize}
\end{defn}

Both of the claims in the next lemma follow simply by writing out the definitions and manipulating inverses.

\begin{lem}\label{lem:inv-switch}
Let $(A, +, 0, \beta)$ be a permutative category. 
\begin{enumerate}
\item\label{it:inv-switch-pair} The tuple $(a, a', \eta, \epz)$ is an invertible pair if and only if $(a', a, \epz^{-1}, \eta^{-1})$ is an invertible pair. 
\item\label{it:inv-switch-map} The pair $(f, f')$ is a map of invertible pairs $(a, a', \eta, \epz) \to (b, b', \delta, \gamma)$ if and only if $(f', f)$ is a map of invertible pairs  
\[
(a', a, \epz^{-1}, \eta^{-1}) \to (b', b, \gamma^{-1}, \delta^{-1}).
\]
\end{enumerate}
\end{lem}

\begin{lem}\label{lem:Ainv-gpd}
  Let $(A, +, 0, \beta)$ be a permutative category with invertible pairs
  $\ul{a} = (a,a', \eta, \epz)$ and $\ul{b} = (b,b', \delta, \gamma)$.
  If $(f, f') \cn \ul{a} \to \ul{b}$ is a map of invertible pairs, then both $f$ and $f'$ are invertible in $A$.
\end{lem}
\begin{proof}
  We will prove that $f'$ is invertible.
  \cref{lem:inv-switch} then implies that $f$ is invertible as well.

  To show that $f'$ is invertible, consider the diagram below.
  \[
    \begin{tikzpicture}[x=24ex,y=9ex]
      \draw[0cell] 
      (0,0) node (a) {b'}
      (.8,0) node (b) {a'+a+b'}
      (2,0) node (c) {a'+b+b'}
      (2.8,0) node (d) {a'}
      (2.8,-1) node (e) {b'}
      (2,-1) node (f) {b'+b+b'}
      
      (.75,-.4) node (1) {\textc{1}}
      (1.75,-.4) node (2) {\textc{2}}
      (2.65,-.4) node (3) {\textc{3}}
      (1.5,-1.9) node (4) {\textc{4}}
      ;
      \draw[1cell] 
      (a) edge node {\eta + \id_{b'}} (b)
      (b) edge node {\id_{a'}+f+\id_{b'}} (c)
      (c) edge node {\id_{a'}+\gamma} (d)
      (d) edge node {f'} (e)
      (c) edge node {f'+\id_b+\id_{b'}} (f)
      (f) edge node['] {\id_{b'}+\gamma} (e)
      (b) edge[swap] node {f'+f+\id_{b'}} (f)
      (a) edge[bend right,swap] node {\delta + \id_{b'}} (f)
      (a) edge[out=-75,in=-135,swap] node {\id_{b'}} (e)
      ;
    \end{tikzpicture}
  \]
  Region $\textc{1}$ commutes by the first map of invertible pairs axiom \cref{eq:inv-obj-map}.
  Regions $\textc{2}$ and $\textc{3}$ commute by the functoriality of $+$.
  Region $\textc{4}$ commutes by the second triangle identity for invertible pairs \cref{eq:aa'-triang}.
  Thus the composite along the top is a right inverse for $f'$, and a similar calculation shows it is a left inverse as well.
  Therefore $f'$ is invertible, completing the proof.
\end{proof}

\begin{lem}\label{lem:iso-to-mapofinv}
  Let $(A, +, 0, \beta)$ be a permutative category with invertible pairs $(a, a', \eta, \epz)$ and $(b, b', \delta, \gamma)$.
  If $f \cn a \cong b$ is an isomorphism in $A$, then there exists a unique isomorphism $f^{\dagger}\cn a' \cong b'$ such that $(f, f^{\dagger})$ is a map of invertible pairs $(a, a', \eta, \epz) \to (b, b', \delta, \gamma)$.
\end{lem}
\begin{proof}
  Let $f^{\dagger}$ be defined as the composite
  \[
    a' \fto{\delta + \id_{a'}} b'+b+a' \fto{\id_{b'} + f^{-1} + \id_{a'}} b'+a+a' \fto{\id_{b'} + \epz} b'.
  \]
  The diagram below shows that the first axiom \cref{eq:inv-obj-map} holds, and the second is similar.
  \[
    \begin{tikzpicture}[x=14ex,y=8ex]
      \draw[0cell] 
      (0,0) node (a) {0}
      (a)++(0,1) node (b) {a'+a}
      (b)++(1,1) node (c) {b'+b+a'+a}
      (c)++(2.5,0) node (d) {b'+a+a'+a}
      (d)++(1,-1) node (e) {b'+a}
      (e)++(0,-1) node (f) {b'+b}
      
      (d)++(-35:.7) node (1) {\textc{1}}
      (1)++(-1.2,-.1) node (2) {\textc{2}}
      (2)++(-1.8,-.2) node (3) {\textc{3}}
      ;
      \draw[1cell] 
      (a) edge node {\eta} (b)
      (b) edge[] node {\delta+ \id_{a'+a}} (c)
      (c) edge[] node {\id_{b'}+f^{-1}+\id_{a'+a}} (d)
      (d) edge[bend left=30] node {\id_{b'} +\epz+ \id_a} (e)
      (e) edge[] node {\id_{b'}+f} (f)
      
      (d) edge[bend right=13,swap] node {\id_{b'}+ \id_a + \eta^{-1}} (e)
      (c) edge[swap] node[pos=.3] {\id_{b' + b} + \eta^{-1}} (f)
      (a) edge node {\delta} (f)
      ;
    \end{tikzpicture}
  \]
  Region $\textc{1}$ commutes by the triangle identities \cref{eq:aa'-triang} for $(a,a',\eta,\epz)$.
  Regions $\textc{2}$ and $\textc{3}$ commute by functoriality of $+$ and, in the latter case, cancellation of $\eta$ with $\eta^{-1}$.

  To prove uniqueness, assume that $g \cn a' \to b'$ makes the second diagram \cref{eq:inv-obj-map} for a map of invertible pairs commute, and consider the diagram below. 
  \begin{equation}\label{eq:g-fdagger}
    \begin{tikzpicture}[x=30ex,y=9ex,vcenter]
      \draw[0cell] 
      (0,0) node (a) {a'}
      (1,0) node (b) {b'}
      (0,-1) node (d) {b'+b+a'}
      (1,-1) node (e) {b'+b+b'}
      (0,-2) node (f) {b'+a+a'}
      (1,-2) node (g) {b'}
      
      (.5,-.4) node (1) {\textc{1}}
      (1.15,-.4) node (2) {\textc{2}}
      (.5,-1.4) node (3) {\textc{3}}
      ;
      \draw[1cell] 
      (a) edge node {g} (b)
      (a) edge[swap] node {\delta+ \id_{a'}} (d)
      (b) edge[swap] node {\delta+ \id_{b'}} (e)
      (d) edge[] node {\id_{b'+b}+ g} (e)
      (e) edge[swap] node {\id_{b'}+\gamma} (g)
      (d) edge[swap] node {\id_{b'}+ f^{-1} +\id_{a'}} (f)
      (f) edge node {\id_{b'}+ \epz} (g)
      ;
      \draw[1cell, rounded corners]
      (b) -- ($(b)+(.3,0)$) -- node {\id_{b'}} ($(g)+(.3,0)$) -- (g)
      ;
    \end{tikzpicture}
  \end{equation}
  Region $\textc{1}$ commutes by functoriality of $+$.
  Region $\textc{2}$ commutes by the triangle identities \cref{eq:aa'-triang} for $(b,b',\de,\ga)$.
  Region $\textc{3}$ commutes by the assumption on $g$ above.
  The composite around the top and right of \cref{eq:g-fdagger} is $g$, while the composite around the left and bottom is $f^{\dagger}$.
  Therefore $g = f^{\dagger}$, proving uniqueness and finishing the proof.
\end{proof}

\begin{defn}\label{defn:smf-map-inv}
  Let $(A, +, 0, \beta)$ and $(B, +, 0, \beta)$ be permutative categories, and let $F:A \to B$ be a symmetric monoidal functor.
  For each invertible pair $\ul{a} = (a,a',\eta,\epz)$, define $\eta^F$ to be the composite
  \begin{equation}\label{eq:smf-map-inv-eta}
    0 \fto{F_0} F(0) \fto{F\eta} F(a'+a) \fto{F_2^{-1}} F(a') + F(a),
  \end{equation}
  and define $\epz^F$ to be the composite
  \begin{equation}\label{eq:smf-map-inv-epz}
    F(a) + F(a') \fto{F_2} F(a+a') \fto{F\epz} F(0) \fto{F_0^{-1}} 0.
  \end{equation}
\end{defn}
\begin{lem}\label{lem:smf-map-inv}
  In the context of \cref{defn:smf-map-inv}, the following statements hold.
  \begin{enumerate}
  \item\label{it:smf-i} Suppose that $\ul{a} = (a, a', \eta, \epz)$ is an invertible pair in $A$.
    The quadruple $F\ul{a} = \bigl( F(a), F(a'), \eta^F, \epz^F\bigr)$ is an invertible pair in $B$.
  \item\label{it:smf-ii} Suppose that $\ul{a} = (a, a', \eta, \epz)$ and $\ul{b} = (b, b', \delta, \gamma)$ are invertible pairs in $A$, and $(f, f')$ is a map of invertible pairs $\ul{a} \to \ul{b}$.
    Then the pair $(Ff, Ff')$ is a map of invertible pairs $F\ul{a} \to F\ul{b}$ in $B$.
  \end{enumerate}
  With the assignments above, $F$ induces a functor
  \[
    \pinv{F} \cn \pinv{A} \to \pinv{B}.
  \]
\end{lem}
\begin{proof}
  Using \cref{eq:smf-map-inv-eta,eq:smf-map-inv-epz}, the data of $F\ul{a}$ in \cref{it:smf-i} constitute an invertible pair by straightforward application of the monoidal functor axioms for $F$ and the invertible pair axioms for $\ul{a}$.
  The proof of \cref{it:smf-ii} follows from \cref{eq:smf-map-inv-eta,eq:smf-map-inv-epz} together with naturality of $F_2$.
  The functoriality of $\pinv{F}$ is an immediate consequence of the functoriality of $F$ and the componentwise composition in $\pinv{A}$ and $\pinv{B}$ (\cref{defn:Ainv}).
  These observations finish the proof.
\end{proof}

\begin{prop}\label{prop:pinv-2fun}
  The assignment $(A, +, 0, \beta) \mapsto \pinv{A}$ is the function on objects of a 2-functor $\pinv{(-)} \cn \permcat \to \cat$.
\end{prop}
\begin{proof}
  The assignment $F \mapsto \pinv{F}$ is defined in \cref{lem:smf-map-inv}.
  We now define the action on 2-cells and check the axioms for a 2-functor.
  For this purpose, suppose given permutative categories $A$, $B$, and $C$ together with symmetric monoidal functors and monoidal transformations as shown here.
  \begin{equation}\label{eq:pinv-2fun-data}
    \begin{tikzpicture}[x=20ex,y=15ex,vcenter]
      \draw[0cell] 
      (0,0) node (a) {A}
      (a)++(1,0) node (b) {B}
      (b)++(1,0) node (c) {C}
      ;
      \draw[1cell] 
      (a) edge[bend left=60] node (f) {F} (b)
      (a) edge node (f2) {H} (b)
      (a) edge[bend right=60] node['] (f3) {K} (b)
      (b) edge[bend left=30] node {G} (c)
      (b) edge[bend right=30] node['] {L} (c)
      ;
      \draw[2cell] 
      node[between=f and f2 at .5, rotate=-90, 2label={above,\phi}] {\Rightarrow}
      node[between=f2 and f3 at .5, rotate=-90, 2label={above,\psi}] {\Rightarrow}
      node[between=b and c at .5, rotate=-90, 2label={above,\mu}] {\Rightarrow}
      ;
    \end{tikzpicture}
  \end{equation}
  Throughout, suppose that $\ul{a} = (a, a', \eta, \epz) \in \pinv{A}$ is an invertible pair.

  The natural transformation $\pinv{\phi}$ is defined by components
  \begin{equation}\label{eq:def-phi-inv}
    \pinv{\phi}_{\ul{a}} = (\phi_a, \phi_{a'}).
  \end{equation}
  These components satisfy the two axioms of \cref{eq:inv-obj-map} by naturality of $\phi$ at $\eta$ and $\epz$, respectively.
  Naturality of $\pinv{\phi}$ follows from that of $\phi$.

  We now check that these definitions give a 2-functor $\permcat \to \cat$.     \begin{description}
  \item[Identity 2-cells] For the identity monoidal transformation, $\id_F$, the components of $\pinv{(\id_F)}$ are $(\id_{Fa}, \id_{Fa'})$.
    Thus, each component is the identity map of invertible pairs, as required.
  \item[Vertical composition of 2-cells] For a vertical composite, we have the following components at $\ul{a}$; the equalities hold by definition and by componentwise composition (\cref{defn:Ainv}):
    \[
      \pinv{(\psi\circ \phi)}_{\ul{a}}
      =  \bigl( (\psi\circ \phi)_a, (\psi\circ \phi)_{a'} \bigr)
      =  \bigl( \psi_a \circ \phi_a, \psi_{a'} \circ \phi_{a'} \bigr)
      = \bigl( \pinv{\psi} \circ \pinv{\phi} \bigr)_{\ul{a}}.
    \]
  \item[Identity 1-cells] For the identity symmetric monoidal functor, $\id_{A}$, the definitions \cref{eq:smf-map-inv-eta,eq:smf-map-inv-epz} show that $\pinv{(\id_A)} = \id_{(\pinv{A})}$ as required.
  \item[Horizontal composition of 1-cells] To confirm $\pinv{(G \circ F)} = \pinv{G} \circ \pinv{F}$, the only nontrivial verifications are $(\eta^F)^G = \eta^{G \circ F}$ and $(\epz^F)^G = \epz^{G \circ F}$.
    By definition in \cref{eq:smf-map-inv-eta}, $(\eta^F)^G$ is given as follows, with the penultimate equality being that of monoidal and unit constraints for a composition of monoidal functors:
    \begin{align*}
      (\eta^F)^G & = G_2^{-1} \circ G(\eta^{F}) \circ G_0 \\
      & = G_2^{-1} \circ G(F_2^{-1} \circ F\eta \circ F_0) \circ G_0 \\
      & = G_2^{-1} \circ G(F_2^{-1}) \circ (G \circ F)\eta \circ G(F_0) \circ G_0 \\
      & = (G \circ F)_2^{-1} \circ (G \circ F)\eta \circ (G \circ F)_0 \\
      & = \eta^{G \circ F}.
    \end{align*}
    The analogous equality holds for $\epz$ by \cref{eq:smf-map-inv-epz}, and these calculations show that $\pinv{(-)}$ preserves composition of 1-cells.
  \item[Horizontal composition of 2-cells] We verify componentwise at $\ul{a} \in \pinv{A}$:
    \[
      \pinv{(\mu * \phi)}_{\ul{a}}
      = \bigl( \mu_{Ha} \circ G( \phi_a )\,,\, \mu_{Ha'} \circ G( \phi_{a'} )\bigr)
      = \bigl( \pinv{\mu}*\pinv{\phi} \bigr)_{\ul{a}}.
    \]
  \end{description}
  This completes the verification that $\pinv{(-)}$ preserves all identities and composition operations, and is therefore a 2-functor.
\end{proof}

\begin{prop}\label{prop:pi-2nat}
Let $(A, +, 0, \beta)$ be a permutative category. Then there is a functor $\pi \cn \pinv{A} \to A$ defined by
\begin{equation}\label{eq:def-pi}
\pi(a, a', \eta, \epz) = a, \quad \pi(f, f') = f
\end{equation}
that is full on isomorphisms and faithful.
These functors are the components of a 2-natural transformation $\pinv{(-)} \to U$, where $U \cn \permcat \to \cat$ is the underlying category 2-functor.
\end{prop}
\begin{proof}
  Since composition and identities for maps of invertible pairs are defined componentwise (\cref{defn:Ainv}), functoriality of $\pi$ is immediate from the definition in \cref{eq:def-pi}.
  By \cref{lem:Ainv-gpd}, every map of invertible pairs has components that are isomorphisms.
  \cref{lem:iso-to-mapofinv} then shows that $\pi$ is full on isomorphisms (by the existence part of \cref{lem:iso-to-mapofinv}) and faithful (by the uniqueness part of \cref{lem:iso-to-mapofinv}).
  The final claim, that the various functors $\pi_A \cn \pinv{A} \to A$ are 2-natural in $A$ with respect to symmetric monoidal functors $A \to B$, is equivalent to the two equalities
  \begin{equation}\label{eq:nat-pi}
    UF \circ \pi_A = \pi_B \circ \pinv{F}
    \andspace
    U\al * \pi_A = \pi_B * \pinv{\al}
  \end{equation}
  for a symmetric monoidal functor $F \cn A \to B$ and monoidal natural transformation $\al \cn F \to G$, respectively.
  The equalities \cref{eq:nat-pi} are both immediate from the definitions of: $\pi$ in  \cref{eq:def-pi}, $\pinv{F}$ in \cref{lem:smf-map-inv}~\cref{it:smf-i}, and $\pinv{\phi}$ in \cref{eq:def-phi-inv}.
  This observation finishes the proof.
\end{proof}

\section{Calculations with invertible pairs}
\label{sec:calcs}

This section will handle key computational results concerning invertible pairs in a permutative category.
We highlight three such calculations.
The first is \cref{lem:even-perms}, which shows that even permutations of an object of the form $k \cdot a$, for $a$ invertible, are necessarily the identity.
The other key results in this section concern three automorphisms of the unit object that we call the Figure Eight, Figure C, and Figure H in \cref{defn:fet-figC-figH}.
Our second calculation, \cref{lem:fet-figC-figH}, shows that these three automorphisms are always equal.
Our third calculation shows that the Figure Eight is its own inverse in \cref{lem:64=1}.

We begin by proving the following Eckmann-Hilton equalities, as they will be useful throughout.
\begin{lem}[Eckmann-Hilton]\label{lem:EH}
  Let $(A, +, 0, \beta)$ be a permutative category.
  \begin{enumerate}
  \item\label{it:EH1} If $f \cn 0 \to 0$ is any endomorphism of the unit, and $g \cn a \to b$ is any morphism, then
    \[
      g \circ (f + \id_a) = f+g = (f + \id_b) \circ g.
    \]
    In the special case when $a=0$, then $g \circ f = (f + \id_b) \circ g$, and in the special case when $b=0$, then $g \circ (f + \id_a) = f \circ g$.
  \item\label{it:EH2} If $f, g \cn 0 \to 0$ are two endomorphisms of the unit object, then $f \circ g = f+ g$.
    Consequently, $A(0,0)$ is a commutative monoid under composition.
  \end{enumerate}
\end{lem}
\begin{proof}
  Since $A$ is a permutative category, for any morphism $g \cn a \to b$, we have the equality $g = g + \id_0 = \id_0 + g$ using the strictness of the monoidal unit and naturality of the equalities $a = a+0 = 0+a$.
  The displayed equalities in \cref{it:EH1} follow from the functoriality of $+$ as shown below:
  \begin{equation}\label{eq:EH}
    \begin{array}{rcl}
      g \circ (f + \id_a) &= & (\id_0 + g) \circ (f + \id_a) \\
                          & = & f+g \\
                          & = & (f + \id_b) \circ (\id_0 + g) \\
                          & = & (f + \id_b) \circ g.
    \end{array}
  \end{equation}
  The stated special cases with $a = 0$ or $b=0$ hold by strictness of the monoidal unit, and \cref{it:EH2} follows by combining the two special cases.
\end{proof}

\begin{cor}\label{cor:homs-comm-monoid}
  Let $(A, +, 0, \beta)$ be a permutative category. If $(a, a', \eta, \epz)$ is an invertible pair in $A$, then $A(0,0) \cong A(a,a)$ as monoids under composition, and therefore $A(a,a)$ is a commutative monoid.
\end{cor}
\begin{proof}
  Recall from \cref{rmk:not-smf} that the translation functor $a+\qmarg$ induces a homomorphism of monoids with respect to composition
  \begin{equation}\label{eq:aqmark-hom}
    a+\qmarg\cn A(0,0) \to A(a,a).
  \end{equation}
  The monoid $A(0,0)$ is commutative by the Eckmann-Hilton \cref{lem:EH}.
  Since $a+\qmarg$ is an equivalence (\cref{rmk:inverse-equivs}), the homomorphism \cref{eq:aqmark-hom} is an isomorphism and the result follows.
\end{proof}

Recall from \cref{notn:braces} that $\SS G$ is the free permutative category on a set $G$, and $\S = \SS\{x\}$ is free on a single generator $x$.
\begin{notn}[Characteristic functors]\label{notn:chi}
  For each permutative category $A$, and each object $a \in A$, there is a unique strict symmetric monoidal functor
  \begin{equation}\label{eq:chia}
    \begin{aligned}
      \chi_a \cn \S & \to A\\
      x & \mapsto a
    \end{aligned}
  \end{equation}
  determined by the free-forgetful adjunction \cref{eq:SS-free-forget} of the 2-monad $\SS$.
  We call $\chi_a$ the \emph{characteristic functor} of $a$.
\end{notn}

Recall from \cref{thm:permcat-coherence-one-obj} that, for a natural number $n$ and permutation $\sigma \in \Sigma_n$, there is an automorphism $[\sigma, 1] \cn n \cdot x \to n \cdot x$ in $\S$ with underlying permutation $\sigma$.
Then given an object $a$ of a permutative category $A$, the strict symmetric monoidal functor $\chi_a \cn \S \to A$ \cref{eq:chia} determines a morphism $\chi_a(\si)\cn n\cdot a \to n\cdot a$.
\begin{prop}\label{lem:even-perms}
  Let $(A, +, 0, \beta)$ be a permutative category with an invertible pair $(a, a', \eta, \epz)$, and let $n$ be a natural number.
  If $\si \in \Sigma_n$ is an even permutation, then $\chi_a(\si) = \id_{n\cdot a}$.
\end{prop}
\begin{proof}
  For an object $x \in A$, write $\Aut(x)$ for the group of automorphisms of $x$ in $A$.
  By functoriality of $\chi_a$ and the fact that $\S$ is a groupoid, the function on morphism sets
  \[
    \chi_a\cn \S(n \cdot x,n \cdot x) \to A(n\cdot a,n\cdot a)
  \]
  induces a group homomorphism (that we denote with the same name)
 \[
    \Sigma_n \cong \S(n \cdot x,n \cdot x) \to \Aut(n\cdot a),
  \] 
  where the first isomorphism is that of \cref{eq:Snn-Sigman}. Since the target is abelian by \cref{cor:homs-comm-monoid}, this composite must factor through $\Sigma_n^{\mathrm{ab}}$, the abelianization, and so the kernel must contain all even permutations.
  Therefore, $\chi_a$ sends each even permutation in $\Sigma_n$ to the identity morphism of $\Aut(n\cdot a)$.
\end{proof}

\begin{rmk}\label{rmk:even-perms}
  \cref{lem:even-perms} conveys the core idea that coherence for an invertible object $a$ is governed by the \emph{sign} of underlying permutations.
  Although not used directly in our further calculations below, it is fundamental.
  Note, however, that the statement of \cref{lem:even-perms} concerns only those morphisms of $A$ that are in the image of $\chi_a \cn \S \to A$.
  In order to extend \cref{lem:even-perms} to morphisms involving both $a$ and its inverse $a'$, one needs a full coherence result such as \cref{thm:free-inv-coherence}.
  Beyond the further calculations in this section, \cref{example:addl,rmk:additional-comps} contain relevant additional discussion.
\end{rmk}

Several of the calculations below use the following observation about naturality of the symmetry and strictness of the monoidal unit in a permutative category.
\begin{rmk}\label{rmk:nat-beta-0}
  Suppose $A$ is a permutative category with objects $x,y,z\in A$ and morphisms $f\cn 0 \to x$ and $g\cn y \to 0$.
  Strictness of the monoidal unit, naturality of $\beta$, and the identities $\beta_{0,z} = \id_z = \beta_{z,0}$, for $z \in A$, imply that the following diagrams commute.
  \begin{equation}\label{eq:nat-beta0}
    \begin{tikzpicture}[x=15ex,y=8ex, vcenter]
      \draw[0cell] 
      (0,0) node (xz) {x + z}
      (xz)++(0,-1) node (zx) {z + x}
      (xz)++(-1,-.5) node (z) {z}
      ;
      \draw[1cell] 
      (xz) edge node {\beta_{x,z}} (zx)
      (z) edge node[pos=.7] {f + \id_z} (xz)
      (z) edge['] node[pos=.7] {\id_z + f} (zx)
      ;
      \draw[0cell] 
      (1,0) node (yz) {y + z}
      (yz)++(0,-1) node (zy) {z + y}
      (yz)++(1,-.5) node (z) {z}
      ;
      \draw[1cell] 
      (yz) edge['] node {\beta_{y,z}} (zy)
      (yz) edge node[pos=.3] {g + \id_z} (z)
      (zy) edge['] node[pos=.3] {\id_z + g} (z)
      ;
    \end{tikzpicture}
    \vspace{-1pc} 
  \end{equation}
  \ 
\end{rmk}

\begin{defn}\label{defn:fet-figC-figH}
  Let $(A, +, 0, \beta)$ be a permutative category, and let $(a,a', \eta, \epz)$ be an invertible pair in $A$.
  We introduce notation for the following standard composites.
  In each case, the name and notation are mnemonics for the corresponding string diagram; we leave drawing these as an exercise for the reader.
  \begin{description}
  \item[Figure Eight]
  \begin{equation}\label{notn:8}
    \begin{tikzpicture}[x=11ex,y=10ex,vcenter]
      \draw[0cell] 
      (0,0) node (a) {0}
      (a)++(.75,-.5) node (b) {a' + a}
      (b)++(1.5,0) node (c) {a + a'}
      (c)++(.75,.5) node (d) {0}
      ;
      \draw[1cell] 
      (a) edge['] node {\eta} (b)
      (b) edge node {\beta_{a',a}} (c)
      (c) edge['] node {\epz} (d)
      (a) edge[bend left=12] node {\fet_a} (d)
      ;
      \draw (-.7,0) node {\ } (3.7,0) node {\ };
    \end{tikzpicture}
  \end{equation}
  \item[Figure C]
  \begin{equation}\label{notn:figC}
    \begin{tikzpicture}[x=11ex,y=10ex,vcenter]
      \draw[0cell] 
      (0,0) node (a) {0}
      (a)++(0,-.75) node (b) {a' + a}
      (b)++(.25,-.75) node (b') {a' + a' + a + a}
      (b')++(2.5,0) node (c') {a' + a' + a + a}
      (c')++(.25,.75) node (c) {a' + a}
      (c)++(0,.75) node (d) {0}
      ;
      \draw[1cell] 
      (a) edge['] node {\eta} (b)
      (b) edge[',pos=.3] node {1 + \eta + 1} (b')
      (b') edge node {1 + 1 + \beta_{a,a}} (c')
      (c') edge[',pos=.7] node {1 + \eta^\inv + 1} (c)
      (c) edge['] node {\eta^\inv} (d)
      (a) edge[bend left=12] node {\figC_a} (d)
      ;
      \draw (-.7,0) node {\ } (3.7,0) node {\ };
    \end{tikzpicture}
  \end{equation}
  \item[Figure H]
  \begin{equation}\label{notn:figH}
    \begin{tikzpicture}[x=11ex,y=10ex,vcenter]
      \draw[0cell] 
      (0,0) node (a) {0}
      (a)++(.25,-.75) node (b) {a' + a + a + a'}
      (b)++(2.5,0) node (c) {a' + a + a + a'}
      (c)++(.25,.75) node (d) {0}
      ;
      \draw[1cell] 
      (a) edge[',pos=.3] node {\eta + \epz^\inv} (b)
      (b) edge node {1 + \beta_{a,a} + 1} (c)
      (c) edge[',pos=.7] node {\eta^\inv + \epz} (d)
      (a) edge[bend left=12] node {\figH_a} (d)
      ;
      \draw (-.7,0) node {\ } (3.7,0) node {\ };
    \end{tikzpicture}
    \vspace{-1pc} 
  \end{equation}
\end{description}
\ 
\end{defn}

\begin{rmk}\label{rmk:8-trace}
  The Figure Eight, $\fet_a$ \cref{notn:8}, for an invertible object $a$ is also known as the \term{trace} of the identity $\id_a$ or the \term{Euler characteristic} of $a$.
  See, e.g., \cite{JSV1996Traced,Dug2014Coherence,PS2014Traces} for further discussion of traces in symmetric monoidal categories, including key examples that justify the terminology.
\end{rmk}

\begin{convention}\label{conv:ol-pm-cat}
  In this section, we use the following notational conventions where they simplify or clarify calculations.
  Each use of these conventions will include a reference back to this explanation.
  \begin{enumerate}
  \item\label{conv-it:ol} In diagrams where space is limited, or in calculations relevant to such diagrams, the inverse of an isomorphism $f\cn x \iso y$ may also be denoted with a bar, $\ol{f} = f^\inv$.
  \item\label{conv-it:pm} For diagrams that involve a single invertible pair $\ul{a} = (a, a', \eta, \epz)$, we may write $+$ and $-$ for the objects $a$ and $a'$, respectively.
    In such cases, we always use concatenation for the monoidal sum, as described in \cref{rmk:concat}.
  \item\label{conv-it:perm} For monoidal sums of more than two objects, we sometimes denote symmetry isomorphisms by their corresponding permutations of summands, numbered from left to right.
  \end{enumerate}
  As an example, for an invertible pair $\ul{a} = (a,a',\eta,\epz)$, the composite
  \[
    a' + a + a' + a \fto{\id_{a'} + \beta_{a+a',a}} a' + a + a + a'
    \fto{\eta^\inv + \id_a + \id_{a'}} a + a'
  \]
  could be denoted
    \[
    -+-+ \fto{(2 \ 3 \ 4)} -++- \fto{\oleta 1 1} +-.
  \]
\end{convention}

The following calculation shows that the three automorphisms of \cref{defn:fet-figC-figH} are equal.
\begin{lem}\label{lem:fet-figC-figH}
  Let $(A, +, 0, \beta)$ be a permutative category, and let $(a,a', \eta, \epz)$ be an invertible pair in $A$.
  The three automorphisms $\fet_a$ \cref{notn:8}, $\figC_a$ \cref{notn:figC}, and $\figH_a$ \cref{notn:figH} are equal as automorphisms of $0$.
\end{lem}
\begin{proof}
  We obtain the desired equalities by commutativity of the following diagram, which uses \cref{conv:ol-pm-cat} and is explained below.
  \begin{equation}\label{eq:fet-figH-diagram}
    \begin{tikzpicture}[x=14ex,y=9ex,vcenter]
      \def\amin{-}
      \def\aplu{+}
      \draw[0cell=.8] 
      (0,0) node (a) {0}
      (a)++(0,-1) node (b) {\amin \aplu}
      (b)++(0,-1) node (c) {\aplu \amin}
      (c)++(0,-1) node (d) {0}
      (a)++(-.7,-.4) node (ax) {\amin \aplu}
      (b)++(-.7,0) node (bx) {\amin \aplu \amin \aplu}
      (bx)++(-1,0) node (by) {\amin \amin \aplu \aplu}
      (by)++(-.7,-.5) node (L) {\amin \amin \aplu \aplu}
      (c)++(-.7,0) node (cx) {\amin \aplu \aplu \amin}
      (cx)++(-1,0) node (cy) {\amin \aplu \amin \aplu}
      (d)++(-.7,.4) node (dx) {\amin \aplu}
      (L)++(0,-1.1) node (dy) {\amin \aplu}
      (b)++(.8,0) node (bz) {\amin \aplu \aplu \amin}
      (c)++(.8,0) node (cz) {\aplu \amin \aplu \amin}
      (cz)++(.8,0) node (R) {\amin \aplu \aplu \amin}
      (d)++(.8,.5) node (dz) {\aplu \amin}
      ;
      \draw[1cell=.7] 
      (a) edge node {\eta} (b)
      (b) edge node {\beta} (c)
      (c) edge node {\epz} (d)
      (a) edge['] node[pos=.6] {\eta} (ax)
      (ax) edge[',bend right] node {\oleta} (a)
      (ax) edge['] node {1 \eta 1} (by)
      (by) edge['] node {(3\; 4)} (L)
      (L) edge['] node {1 \oleta 1} (dy)
      (dy) edge['] node {1} (dx)
      (dx) edge node[pos=.4] {\oleta} (d)
      (by) edge['] node {(2\; 3\; 4)} (bx)
      (bx) edge['] node {\oleta 1 1} (b)
      (cy) edge node {(2\; 3\; 4)} (cx)
      (cx) edge['] node {\oleta 1 1} (c)
      (cy) edge[', bend right=10] node {1 1 \oleta} (dy)
      (dy) edge[', bend right=10] node[pos=.7] {1 1 \eta} (cy)
      (cy) edge['] node {1 \epz 1} (dx)
      (L) edge node[pos=.3] {(2\; 3\; 4)} (cy)
      (by) edge node {(2\; 3)} (cy)
      (ax) edge node {1 1 \eta} (bx)
      (bx) edge node {(3\; 4)} (cx)
      (cx) edge node {1 1 \epz} (dx)
      (a) edge node {\eta \olepz} (bz)
      (bz) edge node {(2\; 3)} (R)
      (R) edge node {\oleta 1 1} (dz)
      (dz) edge['] node {\epz} (d)
      (b) edge['] node {1 1 \olepz} (bz)
      (c) edge node {1 1 \olepz} (cz)
      (cz) edge node {(1\; 3\; 2)} (R)
      (bz) edge['] node {(1\; 2)} (cz)
      (cz) edge['] node {1 \oleta 1} (dz)
      (c) edge[', bend right=10] node {1} (dz)
      ;
      \draw[1cell=.7, rounded corners]
      (a) -- ++(-2.8,0) -- node[',pos=.3] {\figC} ++(0,-3) -- (d)
      ;
      \draw[1cell=.7, rounded corners]
      (a) -- ++(2,0) -- node[pos=.3] {\figH} ++(0,-3) -- (d)
      ;
      \draw[0cell=.7]
      (bx)++(135:.3) node {\textc{n}}
      (cx)++(225:.3) node {\textc{n}}
      (L)++(-65:.4) node {\textc{n}}
      (cz)++(-45:.25) node {\textc{n}}
      ;
      \draw[0cell=.7]
      (bx)++(225:.707) node {\textc{p}}
      (by)++(245:.38) node {\textc{p}}
      (cz)++(60:.45) node {\textc{p}}
      ;
    \end{tikzpicture}
  \end{equation}
  In the above diagram, the left composite is $\figC_a$, the right composite is $\figH_a$, and the vertical composite through the middle is $\fet_a$.
  The regions labeled {\color{cola}[\textsc{n}]} commute by naturality of $\beta$ as in \cref{rmk:nat-beta-0}.
  The regions labeled {\color{cola}[\textsc{p}]} commute by equality of composite permutations.
  The remaining regions commute by functoriality of the monoidal sum; invertibility of $\eta$; and triangle identities for $(\eta,\epz)$ (at left) and $(\olepz,\oleta)$ (at right).
\end{proof}

\begin{cor}[2-Torsion of $\fet$]\label{lem:64=1}
  Let $(A, +, 0, \beta)$ be a permutative category, and let $(a,a', \eta, \epz)$ be an invertible pair in $A$.
  Then $\fet_a \circ \fet_a = \id_0$. 
\end{cor}
\begin{proof}
  It is immediate from the formulas in \cref{defn:fet-figC-figH} that
  \[
    \figC_a \circ \figC_a = \id_0 = \figH_a \circ \figH_a.
  \]
  Therefore, the result holds for $\fet_a$ by \cref{lem:fet-figC-figH}.
\end{proof}

\begin{lem}\label{lem:8-invariant}
  Let $(A, +, 0, \beta)$ be a permutative category, and $\ul{a}=(a, a', \eta, \epz)$ be an invertible pair in $A$.
  The morphism $\mathtt{8}_a$ depends only on the isomorphism class of $a$ as an object of $A$.
\end{lem}
\begin{proof}
  Suppose $b$ is an object of $A$ with an isomorphism $f\cn a \iso b$.
  Following the discussion in \cref{rmk:inverse-equivs,lem:weak-vs-structured-inv} we conclude that $b$ is invertible and is the base object of some invertible pair $\ul{b} = (b,b',\de,\ga)$.
  Then, \cref{lem:iso-to-mapofinv} shows that $f$ extends to an isomorphism of invertible pairs (\cref{defn:inv-obj-map}).
  \[
    (f,f^\dagger)\cn \ul{a} \to \ul{b} \inspace \pinv{A}.
  \]
  The defining properties \cref{eq:inv-obj-map} together with naturality of $\beta$ shows that the following diagram commutes and hence $\fet_a = \fet_b$.
  \[
    \begin{tikzpicture}[x=13ex,y=9ex,vcenter]
      \draw[0cell] 
      (0,0) node (a) {0}
      (a)++(.75,.5) node (b) {a' + a}
      (b)++(1.5,0) node (c) {a + a'}
      (c)++(.75,-.5) node (d) {0}
      (a)++(.75,-.5) node (b') {a' + a}
      (b')++(1.5,0) node (c') {a + a'}
      ;
      \draw[1cell] 
      (a) edge node[pos=.6] {\eta} (b)
      (b) edge node {\beta} (c)
      (c) edge node[pos=.4] {\epz} (d)
      (a) edge node[',pos=.6] {\de} (b')
      (b') edge node['] {\beta} (c')
      (c') edge node[',pos=.4] {\ga} (d)
      (b) edge node {f^\dagger + f} (b')
      (c) edge node['] {f+ f^\dagger} (c')
      ;
      \draw[1cell,rounded corners]
      (a) -- ($(a)+(0,1.0)$) -- node {\fet_a} ($(d)+(0,1.0)$) -- (d)
      ;
      \draw[1cell,rounded corners]
      (a) -- ($(a)+(0,-1.0)$) -- node['] {\fet_b} ($(d)+(0,-1.0)$) -- (d)
      ;
    \end{tikzpicture}
  \]
\end{proof}

\begin{lem}\label{lem:betaaa=8aaa}
  Let $(A, +, 0, \beta)$ be a permutative category, and let $(a,a', \eta, \epz)$ be an invertible pair in $A$.
  Each of the following equalities holds in $A$.
  \begin{align}
    \fet_a &= \fet_{a'}
    &\inspace& A(0,0),\label{it:8a=8a'}\\
    \beta_{a,a} &= \fet_a + 2\cdot a
    &\inspace& A(a+a,a+a),\  \andspace\label{it:betaaa=8aaa} \\
    \beta_{a',a'} &= \fet_a + 2\cdot a'
    &\inspace& A(a'+a',a'+a').\label{it:betaa'a'=8aa'a'} 
  \end{align}
\end{lem}
\begin{proof}
  Note that \cref{it:betaa'a'=8aa'a'} follows from \cref{it:8a=8a',it:betaaa=8aaa}.
  To verify \cref{it:8a=8a'}, the automorphism $\fet_{a'}$ can be defined using the invertible pair $(a',a,{\epz}^\inv, {\eta}^\inv)$ by \cref{lem:inv-switch}.
  Any other invertible pair with base $a'$ will yield the same automorphism $\fet_{a'}$ by \cref{lem:8-invariant}.
  The definition \cref{notn:8} for $\fet_a$ and $\fet_{a'}$, followed by the 2-Torsion of $\fet$ (\cref{lem:64=1}), yields the desired equality as $\fet_{a'} = \fet_a^\inv = \fet_a$.
  
  Now we prove \cref{it:betaaa=8aaa}.
  By the Eckmann-Hilton \cref{lem:EH} and \cref{lem:fet-figC-figH}, it suffices to show $\beta_{a,a} = a + \figH_a + a$. 
  For that purpose, consider the following diagram, which again uses \cref{conv:ol-pm-cat} and is explained below.
  \begin{equation}\label{eq:diagram-betaaa=8aaa}
    \begin{tikzpicture}[x=12ex,y=5ex,vcenter]
      \draw[0cell=.8] 
      (0,0) node (a) {++}
      (a)++(4,0) node (a') {++}
      (a)++(1,-1) node (b) {+-++-+}
      (b)++(-1,-1) node (c) {++}
      (c)++(1,-1) node (d) {+-++-+}
      (d)++(3,0) node (d') {+-++-+}
      (a')++(1,-1) node (c') {++}
      ;
      \draw[1cell=.7] 
      (a) edge node {\beta} (a')
      (a) edge node[pos=.7] {(\olepz 1)(1 \eta)} (b)
      (b) edge['] node[pos=.3] {(1 \oleta)(\epz 1)} (c)
      (c) edge node {(1 \eta)(\olepz 1)} (d)
      (d) edge node {(3 \ 4)} (d')
      (a') edge['] node {(\olepz 1)(1 \eta)} (d')
      (d') edge node[pos=.7] {(1 \oleta)(\epz 1)} (c')
      (a) edge[bend right=10] node {1} (c)
      (b) edge[bend left=10] node {1} (d)
      (a') edge[bend left=10] node {1} (c')
      ;
      \draw[1cell=.7, rounded corners]
      (c) -- ++(0,-1.5) -- node['] {1 \figH 1} ++(5,0) -- (c')
      ;
    \end{tikzpicture}
  \end{equation} 
  The above diagram commutes by definition of $\figH_a$ \cref{notn:figH}, functoriality of the monoidal sum, the triangle identities for $(\olepz,\oleta)$ and $(\eta,\epz)$ (twice each), and invertibility of $\eta$ and $\epz$ (once each).
  Equality of the two outer composites shows $\beta_{a,a} = a + \figH_a + a$, as desired.
\end{proof}

\begin{lem}\label{lem:feightone-equals-eightone}
  Let $(A, +, 0, \beta)$ and $(B, +, 0, \beta)$ be permutative categories, and $F \cn A \to B$ be a symmetric monoidal functor between them.
  Suppose that $\ul{a} = (a,a',\eta,\epz)$ is an invertible pair in $A$.
  Then the following diagram commutes in $B$.
  \begin{equation}\label{eq:feightone-equals-eightone}
    \begin{tikzpicture}[x=24ex,y=8ex,vcenter]
      \draw[0cell] 
      (0,0) node (a) {0}
      (1,0) node (b) {F(0)}
      (1,-1) node (c) {F(0)}
      (0,-1) node (d) {0}
      ;
      \draw[1cell] 
      (a) edge node {F_0} (b)
      (b) edge[] node {F(\mathtt{8}_a)} (c)
      (a) edge[swap] node {\mathtt{8}_{F(a)}} (d)
      (d) edge[swap] node {F_0} (c)
      ;
    \end{tikzpicture}
  \end{equation}
\end{lem}
\begin{proof}
  Unpacking the definition of $\fet_{Fa}$ in \cref{notn:8} via $\eta^F$ in \cref{eq:smf-map-inv-eta} and $\epz^F$ in \cref{eq:smf-map-inv-epz}, functoriality of $F$ reduces commutativity of \cref{eq:feightone-equals-eightone} to the symmetry axiom \cref{eq:smfunctor} for $F_2$ and $\beta_{a,a'}$.
\end{proof}

\section{The free permutative category on an invertible generator}
\label{sec:flex-mod}

The purpose of this section is to define a permutative category $\P$ and show that it has two key properties: (\textit{i}) $\P$ represents the 2-functor $\pinv{(?)}$ of \cref{prop:pinv-2fun}, and (\textit{ii}) $\P$ has the additional property of being \emph{flexible} with respect to the 2-monad $\SS$ for permutative categories.
The first of these properties, that $\P$ represents $\pinv{(?)}$, justifies calling $\P$ the free permutative category on an invertible generator.
\Cref{defn:ppushouts} constructs $\P$ via a sequence of colimits.
The advantage of such a construction is that representability of $\pinv{(?)}$ (\cref{prop:prepn}) is a straightforward exercise in universal properties.

Our second desideratum, that of flexibility, deserves additional explanation.
In the simplest terms, flexible algebras over a 2-monad are a generalization of the free algebras. 
In the case of a 2-monad $T$ on $\cat$, flexible $T$-algebras can often be described as those algebras whose equations at the level of objects are no more than those required by $T$.
Our interest in proving the flexibility of $\P$ is actually to apply the slightly weaker property of semiflexibility (\cref{defn:semiflex}), a consequence of flexibility.
A semiflexible algebra $X$ is one for which every pseudo-$T$-morphism $F \cn X \to Y$ is isomorphic to a strict $T$-morphism; translating to permutative categories, a semiflexible permutative category $A$ is one such that every symmetric monoidal functor $F \cn A \to B$ is isomorphic to a \emph{strict} symmetric monoidal functor $\hat{F} \cn A \to B$.

We begin this section with an overview of pseudomorphism classifiers, flexible and semiflexible algebras, and the Lack model structure on $\permcats$.
This framework, and some standard model category arguments, prove that $\P$ is flexible (and therefore also semiflexible) in \cref{prop:pflex}.
We conclude with a strictification result, \cref{lem:pflex-id}, that will be a key ingredient in the proof of \cref{thm:ZZ-equiv-P1}.

Since $\SS$ is the 2-monad for permutative categories (\cref{thm:permcat-monadic}), the inclusion \cref{eq:incl-j}
\[
  j\cn \permcats \to \permcat
\]
is the inclusion of the 2-category of $\SS$-algebras, strict algebra morphisms, and algebra 2-cells among the more general $\SS$-algebra \emph{pseudo}-morphisms---that is, the strong symmetric monoidal functors.
\begin{defn}[Pseudomorphism Classifier]
  A \term{pseudomorphism classifier} for $\permcat$ is a left 2-adjoint $Q \dashv j$ as shown below.
  \begin{equation}\label{eq:Qi-adj}
    \begin{tikzpicture}[x=15ex,y=8ex,vcenter]
      \draw[0cell] 
      (0,0) node (x) {\permcat}
      (1,0) node (y) {\permcats}
      ;
      \draw[1cell] 
      (x) edge[bend left=12,transform canvas={yshift=.7mm}] node (L) {Q} (y) 
      (y) edge[bend left=12,transform canvas={yshift=-.7mm}] node (R) {j} (x) 
      ;
      \draw[2cell] 
      node[between=L and R at .5] {\bot}
      ;
    \end{tikzpicture}
  \end{equation}
  The unit $\zeta\cn 1 \to j Q$ has components that are symmetric monoidal functors
  \begin{equation}\label{eq:zeta}
    \zeta_X \cn X \to j Q X \forspace X \in \permcat.
  \end{equation}
  The counit $\delta\cn Q j \to 1$ has components that are \emph{strict} symmetric monoidal functors
  \begin{equation}\label{eq:delta}
    \delta_Y \cn Q j Y \to Y \forspace Y \in \permcats.
  \end{equation}
  The 2-adjunction $Q \dashv j$ then yields the following 2-natural isomorphism of categories, for permutative categories $X$ and $Y$:
  \begin{equation}\label{eq:Qj-free-forget}
    \permcats(Q X, Y) \cong \permcat(X, j Y).
  \end{equation}
\end{defn}

\begin{thm}[{\cite[Theorem~3.13]{BKP1989Two},~\cite[Section~4.1]{Lac02Codescent}}]\label{thm:Qexists}
  The 2-monad $\SS$ for permutative categories has a pseudomorphism classifier $Q \dashv j$.
\end{thm}

\begin{defn}\label{defn:flex}
  A permutative category $A$ is \term{flexible} if the counit $\de_A$ \cref{eq:delta} is a \term{surjective} equivalence in $\permcats$: that is, $\de_A$ is part of an equivalence such that the reverse functor is a section.
\end{defn}

\begin{defn}\label{defn:semiflex}
  A permutative category $A$ is \term{semiflexible} if the inclusion $j \cn \permcats \to \permcat$ induces 2-natural equivalences of categories
  \begin{equation}\label{eq:semiflex}
    \permcats(A,X) \fto[\hty]{j} \permcat(j A,j X)
  \end{equation}
  for all permutative categories $X$.
\end{defn}

\begin{prop}\label{prop:sf-imp-f}
  Let $A$ be a permutative category.
  \begin{enumerate}
  \item\label{it:free-flex} If $A$ is free as a permutative category (i.e., $A \cong \SS\{\cC\}$ for some small category $\cC$), then it is flexible \cite[Corollary~5.6]{BKP1989Two}.
  \item\label{it:flex-is-semiflex} If $A$ is flexible, then it is semiflexible \cite[Theorem~4.7]{BKP1989Two}.
  \end{enumerate}
\end{prop}

\begin{rmk}[Quillen model structures]\label{rmk:model-structure}
  Recall that a \term{Quillen model structure} on a category $M$ consists three classes of morphisms, called \emph{fibrations}, \emph{cofibrations}, and \emph{weak equivalences}.
  These classes are required to satisfy several axioms related to factorization, whose specific details will not be used in this work.
  We refer the reader to \cite{hovey1999mc}, \cite{Hir03Model}, or \cite{Lac02Quillen2} for more thorough treatment.

  Our work below will depend only on the following three elementary statements from the theory of model structures.
  \begin{enumerate}
  \item A category $M$ is called a \term{Quillen model category}, or simply a \term{model category}, when it is complete and cocomplete and is equipped with a Quillen model structure.
  \item\label{it:cof-cof} An object $X$ in a model category $M$ is called \term{cofibrant} if the unique morphism from the initial object of $M$ to $X$ is a cofibration in $M$.
    Thus, if $X$ is cofibrant and there is a cofibration $X \to Y$ for some object $Y$, then $Y$ is also cofibrant.
  \item\label{it:cof-po} Suppose that $C \xleftarrow{g} A \fto{f} B$ is a span in a model category $M$, and let $D$ denote its pushout.
    If $g$ is a cofibration, then so is the structure morphism $B \to D$ for the pushout.
    In particular, if $B$ is cofibrant and $g$ is a cofibration, then $D$ is also cofibrant.
  \end{enumerate}
  An explanation of this third property is elementary, but beyond our current scope; it requires further details of the lifting properties that define a Quillen model structure.
  We refer the reader to \cite[Section~7.2.9]{Hir03Model} for careful explanations.
\end{rmk}

\begin{defn}[{\cite[Theorem~2]{JT1991Strong}}]\label{defn:can-mod}
  The \term{canonical model structure} on $\cat$ is the Quillen model structure such that:
  \begin{itemize}
  \item fibrations are isofibrations,
  \item weak equivalences are equivalences of categories, and
  \item cofibrations are injective-on-objects functors.
  \end{itemize}
\end{defn}

\begin{prop}[{\cite[Theorem~5.5]{Lack2007Homotopy}}]\label{prop:lack-model}
  There is a Quillen model structure, called the \term{Lack model structure}, on $\permcats$ transferred from the canonical model structure on $\cat$.
  In this model structure, the weak equivalences are the equivalences of underlying categories and the fibrations are the isofibrations of underlying categories.

  Thus, the free-forgetful adjunction
  \[
    (\SS \dashv U)\cn \cat \lradj \permcats
  \]
  is a Quillen adjunction.
  In particular, if $h\cn C \to D$ is an injective-on-objects functor, then $\SS h\cn \SS C \to \SS D$ is a cofibration in $\permcats$.
\end{prop}

The following is a direct application of \cite[Theorem~5.12]{Lack2007Homotopy} in the context of permutative categories.
\begin{prop}\label{prop:cofibrant-flexible}
  In the Lack model structure on $\permcats$, the cofibrant objects are precisely the flexible permutative categories.
\end{prop}

For the following, recall the 2-dimensional aspect of (2-)colimits, noted in \cref{rmk:permcat-complete-cocomplete}.
\begin{defn}\label{defn:ppushouts}
  Consider a set with two objects, $\anb$, and use the following notation for small categories with the same objects and the indicated non-identity morphisms.
  \begin{center}
    \begin{tabular}{c|l}
      Notation & Description\\
      \hline
      $\anb$ & Discrete category\\
      $\aisob$ & Isomorphism between $a$ and $b$\\
      $\atob$ & One non-identity morphism from $a$ to $b$\\
      $\atotob$ & Two distinct non-identity morphisms from $a$ to $b$
    \end{tabular}
  \end{center}
  Then, we use the notation
  \begin{align}
    p \cn & \anb \to \aisob \andspace\label{eq:anb-aisob}\\
    q \cn & \atotob \to \atob\label{eq:atotob-atob}
  \end{align}
  for the unique functors that are identities on objects and have the indicated domain and codomain.

  Now we use the notation above to define a permutative category $\P$ as follows.
  \begin{description}
    \item[Step 1] Let $\S_1$ be the following pushout in $\permcats$, explained below.
      \begin{equation}\label{eq:ppush-1}
        \begin{tikzpicture}[x=16ex,y=8ex,vcenter]
          \draw[0cell] 
          (0,0) node (a) {\SS\anb}
          (2,0) node (b) {\SS\{x,y\}}
          (0,-1) node (c) {\SS\aisob}
          (2,-1) node (d) {\S_1}
          (1.75,-.75) node (po) {\ulcorner}
          ;
          \draw[1cell] 
          (a) edge node {f} node[',scale=.8] {\anb \mapsto \{x+y,0\}} (b)
          (b) edge node {} (d)
          (a) edge[swap] node {\SS p} (c)
          (c) edge[swap] (d)
          ;
        \end{tikzpicture}
      \end{equation}
      The strict symmetric monoidal functor $f$ along the top is uniquely induced by requiring $f(a) = x+y, f(b) = 0$.
      We write $x, y \in \S_1$ for the images of the generating objects $x, y \in \SS\{x,y\}$ under $\SS\{x,y\} \to \S_1$.
      We define $\epz \cn x+y \cong 0$ as the image of the isomorphism $a \cong b$ in $\SS\aisob$ under $\SS\aisob \to \S_1$.
    \item[Step 2] Let $\S_2$ be the following pushout in $\permcats$, explained below.
      \begin{equation}\label{eq:ppush-2}
        \begin{tikzpicture}[x=16ex,y=8ex,vcenter]
          \draw[0cell] 
          (0,0) node (a) {\SS\anb}
          (2,0) node (b) {\S_1}
          (0,-1) node (c) {\SS\aisob}
          (2,-1) node (d) {\S_2}
          (1.75,-.75) node (po) {\ulcorner}
          ;
          \draw[1cell] 
          (a) edge node {f} node[',scale=.8] {\anb \mapsto \{0,y+x\}} (b)
          (b) edge node {} (d)
          (a) edge[swap] node {\SS p} (c)
          (c) edge[swap] (d)
          ;
        \end{tikzpicture}
      \end{equation}
      The strict symmetric monoidal functor $f$ along the top is uniquely induced by requiring $f(a) = 0, f(b) = y+x$.
      We write $x, y \in \S_2$ for the images of the objects $x, y \in \S_1$ under $\S_1 \to \S_2$.
      We define $\eta \cn 0 \cong y+x$ as the image of the isomorphism $a \cong b$ in $\SS\aisob$ under $\SS\aisob \to \S_2$.
    \item[Step 3] Let $\S_3$ be the following pushout in $\permcats$, explained below.
      \begin{equation}\label{eq:ppush-3}
        \begin{tikzpicture}[x=16ex,y=8ex,vcenter]
          \draw[0cell] 
          (0,0) node (a) {\SS\atotob}
          (2,0) node (b) {\S_2}
          (0,-1) node (c) {\SS\atob}
          (2,-1) node (d) {\S_3}
          (1.75,-.75) node (po) {\ulcorner}
          ;
          \draw[1cell] 
          (a) edge node {f}
          node[',scale=.8] {\atotob \mapsto \{(\epz +x)\circ(x +\eta)\,,\, \id_x\}}
          (b)
          (b) edge node {} (d)
          (a) edge[swap] node {\SS q} (c)
          (c) edge[swap] node {} (d)
          ;
        \end{tikzpicture}
      \end{equation}
      The strict symmetric monoidal functor $f$ along the top is uniquely induced by sending one of the arrows $a \to b$ to the composite
      \[
        x \fto{x+ \eta} x+y+x \fto{\epz +x} x
      \]
      and the other to $\id_x$.
    \item[Step 4] Let $\P$ be the following pushout in $\permcats$, explained below.
      \begin{equation}\label{eq:ppush-4}
        \begin{tikzpicture}[x=16ex,y=8ex,vcenter]
          \draw[0cell] 
          (0,0) node (a) {\SS\atotob}
          (2,0) node (b) {\S_3}
          (0,-1) node (c) {\SS\atob }
          (2,-1) node (d) {\P}
          (1.75,-.75) node (po) {\ulcorner}
          ;
          \draw[1cell] 
          (a) edge node {f}
          node[',scale=.8] {\atotob \mapsto \{(y +\epz)\circ(\eta+ y)\,,\, \id_y\}}
          (b)
          (b) edge node {} (d)
          (a) edge[swap] node {\SS q} (c)
          (c) edge[swap] node {} (d)
          ;
        \end{tikzpicture}
      \end{equation}
      The strict symmetric monoidal functor $f$ along the top is uniquely induced by sending one of the arrows $a \to b$ to the composite
      \[
        y \fto{\eta +y} y+x+y \fto{y +\epz} y
      \]
      and the other to $\id_y$.
  \end{description}
  We will write $\P$ or $\P\{x\}$ for the pushout constructed above, and call this the
  \term{free permutative category on an invertible generator $x$}.
  Its universal property is developed in \cref{prop:prepn}.

  For a finite set $G$, we will write
  \begin{equation}\label{equation:PPG-defn}
    \PP G = \coprod_{g \in G} \P\{g\},
  \end{equation}
  with the coproduct in $\permcats$ and the inverse to $g$ generically denoted $g'$.
  We call $\PP G$ the
  \term{free permutative category on a set of invertible generators, $G$}.
\end{defn}

\begin{prop}\label{prop:prepn}
  The permutative category $\P$ represents the 2-functor 
  \[
    \pinv{(-)} \cn \permcats \to \cat.
  \]
  Equivalently, there is an isomorphism of categories
  \begin{equation}\label{eq:prepn}
    \permcats(\P,A) \fto[\iso]{\e} \pinv{A}
  \end{equation}
  that is 2-natural with respect to strict symmetric monoidal functors $f\cn A \to B$.
\end{prop}
\begin{proof}
  Suppose $A$ is a permutative category and define a functor
  \[
    \e\cn \permcats(\P,A) \to \pinv{A}
  \]
  as follows.
  The construction of $\P$ as a sequence of pushouts in \cref{defn:ppushouts} shows that $x \in \P$ is the base object of an invertible pair $(x,y,\eta,\epz)$.
  The data for such an invertible pair is constructed in Steps 1 and 2 of \cref{defn:ppushouts}, and the axioms follows from Steps 3 and 4.
  Therefore, for any symmetric monoidal functor $h\cn \P \to A$, the quadruple $\bigl(h(x),h(y),h(\eta),h(\epz)\bigr)$ defines an invertible pair in $A$.
  Define
  \begin{equation}\label{eq:eh}
    \e(h) = \bigl( h(x), h(y), h(\eta), h(\epz) \bigr).
  \end{equation}
  Likewise, a monoidal transformation $\si\cn h \to k$ between symmetric monoidal functors $h,k\cn \P\to A$ defines a morphism of invertible pairs,
  \[
    (\si_x,\si_y)\cn \e(h) \to \e(k).
  \]
  Define
  \begin{equation}\label{eq:esi}
    \e(\si) = (\si_x,\si_y).
  \end{equation}
  This assignment on morphisms is functorial because identities and composites of monoidal transformations, and of invertible pairs, are determined componentwise (\cref{defn:Ainv,defn:montransf}).

  Now we define an inverse,
  \[
    \einv\cn \pinv{A} \to \permcats(\P,A)
  \]
  as follows.
  Suppose $\ul{a} = (a, a', \de, \ga)$ is an invertible pair in $A$.
  By construction of $\S_1$ as the pushout \cref{eq:ppush-1}, the counit $\ga\cn a+a' \iso 0$ in $A$ determines a unique strict symmetric monoidal functor $\S_1 \to A$ that sends the generating isomorphism of $\SS\aisob$ to $\ga$.
  Likewise, the isomorphism $\de\cn a'+a \iso 0$ in $A$ determines a unique strict symmetric monoidal functor $\S_2 \to A$ extending the previous construction and sending the generating isomorphism of $\SS\aisob$ to the unit isomorphism $\de$.
  The two triangle identities for $\de$ and $\ga$ then determine, in turn, unique extensions to $\S_3$ and finally $\P$.
  Define
  \begin{equation}\label{eq:einva}
    \einv(\ul{a})\cn \P \to A
  \end{equation}
  as the strict symmetric monoidal functor constructed in this way.

  The definition of $\einv$ on morphisms of $\pinv{A}$ uses the 2-dimensional aspect of the pushouts in \cref{defn:ppushouts}, as follows (recall \cref{rmk:permcat-complete-cocomplete}).
  Given a morphism of invertible pairs, $(f,f')\cn \ul{a} \to \ul{b}$, the counits for $\ul{a}$ and $\ul{b}$, respectively, determine two morphisms
  \begin{equation}\label{eq:two-morphisms}
    \S_1 \to A
  \end{equation}
  by the 1-dimensional universal property of the pushout \cref{eq:ppush-1}, as described above.
  Compatibility of $(f,f')$ with the counits \cref{eq:inv-obj-map} then implies, by the 2-dimensional aspect of the pushout \cref{eq:ppush-1}, that $(f,f')$ defines a monoidal transformation between the two morphisms \cref{eq:two-morphisms}.
  Continuing similarly through each of the pushouts defining $\P$, the pair $(f,f')$ determines a monoidal transformation from $\einv(\ul{a})$ to $\einv(\ul{b})$.
  Define
  \begin{equation}\label{eq:einvf}
    \einv(f,f')\cn \einv(\ul{a}) \to \einv(\ul{b})
  \end{equation} 
  to be this monoidal transformation.
  The composite $\e \circ \einv$ is the identity by construction.
  Uniqueness in the universality of each pushout shows that the composite $\einv \circ \e$ is the identity.
  Therefore the assignments \cref{eq:einva,eq:einvf} defining $\einv$ are inverse to those of $\e$, and hence $\einv$ is also functorial.
  This completes the proof that $\e$ is an isomorphism of categories.
\end{proof}

\begin{rmk}\label{rmk:pic-monadic}
  A permutative category in which every object is invertible is called a \term{Picard category} or \term{Picard groupoid} \cite{JO2012Modeling,Dug2014Coherence,GJO2019Topological,Bra2020Braided}.
  The construction of $\P$ and verification of its universality in \cref{defn:ppushouts,prop:prepn} are the essential steps in the construction of what one would call the free \term{coherent Picard category} 2-monad on $\cat$.
  We do not pursue such a line of research here, as our main goal can be achieved by working in the 2-category of permutative categories.
\end{rmk}

\begin{prop}\label{prop:pflex}
  The permutative category $\P$ constructed in \cref{defn:ppushouts} is flexible.
\end{prop}
\begin{proof}
  By \cref{prop:sf-imp-f}, the free permutative category $\SS \{ x, y \}$ is flexible and therefore cofibrant in the Lack model structure on $\permcats$ by \cref{prop:cofibrant-flexible}.
  Recalling \cref{eq:anb-aisob}, the strict symmetric monoidal functor $\SS p \cn \SS \anb \to \SS \aisob$ is a cofibration in the Lack model structure by \cref{prop:lack-model}.
  Since pushouts of cofibrations are cofibrations (\cref{rmk:model-structure}~\cref{it:cof-po}), the strict symmetric monoidal functor 
  \[
    \SS \{x,y \} \to \S_1
  \]
  in \cref{eq:ppush-1} is a cofibration, and therefore $\S_1$ is cofibrant (\cref{rmk:model-structure}~\cref{it:cof-cof}).
  The same argument, applied iteratively to the pushouts in \cref{eq:ppush-2,eq:ppush-3,eq:ppush-4}, shows that each of $\S_2$, $\S_3$, and $\P$ is cofibrant.
  Therefore $\P$ is flexible, once again by \cref{prop:cofibrant-flexible}.
\end{proof}

\begin{cor}\label{cor:pflex-equiv}
  The inclusion $j \cn \permcats \to \permcat$ induces equivalences
  \[
    \permcats(\P,A) \fto[\hty]{j} \permcat(\P,A)
  \] 
  for all permutative categories $A$.
  In particular, every symmetric monoidal functor $F \cn \P \to A$ is isomorphic, in $\permcat$, to a strict symmetric monoidal functor $\hat{F} \cn \P \to A$.
\end{cor}
\begin{proof}
  \cref{prop:pflex} shows that $\P$ is flexible, and therefore semiflexible by \cref{prop:sf-imp-f}~\cref{it:flex-is-semiflex}.
  This completes the proof by \cref{defn:semiflex}.
\end{proof}

\begin{lem}\label{lem:pflex-id}
  Let $\P$ be the free permutative category on an invertible generator $x$. Suppose
  \[
    F \cn \P \to \P
  \]
  is a symmetric monoidal functor such that there is an isomorphism $F(x) \iso x$ in $\P$.
  Then there is a monoidal natural isomorphism $F \iso \id_\P$.
\end{lem}
\begin{proof}
  By \cref{cor:pflex-equiv}, there is a monoidal natural isomorphism $F \iso \hat{F}$ with $\hat{F} \cn \P \to \P$ strict symmetric monoidal.
  Let $w = \hat{F}(x)$.
  Evaluation at $x$ gives a sequence of isomorphisms in $\P$
  \begin{equation}\label{eq:hatx-Fx-x}
  w = \hat{F}(x) \iso F(x) \iso x.
  \end{equation}
  By \cref{prop:prepn}, $\hat{F}$ corresponds to a unique invertible pair $\ul{w} \in \pinv{\P}$ with base object $w$, and hence \cref{eq:hatx-Fx-x} extends to an isomorphism of invertible pairs $\ul{w} \iso \ul{x}$ in $\pinv{\P}$ by \cref{lem:iso-to-mapofinv}.
  Therefore, we have $\hat{F} \iso \id_\P$ by \cref{prop:prepn}.
  The result then follows by composing the monoidal natural isomorphisms $F \iso \hat{F} \iso \id_\P$.
\end{proof}

\section{The super integers}
\label{sec:superZ}

Here we define and study a permutative category $\Z$ that is a skeletal version of Kapranov's \emph{skeleton of the sign rule}, $1\mh\text{Ab}^{\mathbbm{Z}}$ \cite[Section~3.1]{Kap2021Supergeometry}.
Instead of defining $\Z$ by reference to graded abelian groups, we do so purely algebraically: the objects of $\Z$ are integers, and the only morphisms are automorphisms given by abstract signs $\pm 1_k$.
The rest of the structure of $\Z$ as a permutative category is defined using simple arithmetic of integers, and every object $k$ of $\Z$ is part of a canonical invertible pair $\ul{k}$ \cref{eq:ul-k}.

The main result of this section, \cref{thm:ZZ-equiv-P1}, shows that the invertible pair $\ul{1}$ induces a symmetric monoidal \emph{equivalence} $K \cn \P \to \Z$.
Suggestions of a 2-categorical approach to such a result have been made at various times---including by the authors---but such work has not appeared in the literature to the authors' knowledge.
The utility of such an equivalence $K$ will be explained in \cref{sec:z-comp-2}, where we use it to simplify a variety of computations involving morphisms between invertible objects.

\begin{defn}\label{defn:super-integers}
  The category of \term{super integers}, $\Z$, is defined as follows.
  The set of objects $\ob\Z$ is the set of integers, $\ZZ$.
  For each pair of integers $k$ and $j$, the morphism set $\Z(k,j)$ is empty if $k \neq j$ and $\Z(k,k) = \ZZ^\times$, the multiplicative group of order two, with composition given by multiplication.

  For clarity, we write subscripts on morphisms of $\Z$ to indicate their co/domain, so 
  \begin{equation}\label{eq:Zkk}
    \Z(k,k) = \{ 1_k, -1_k \} \forspace k \in \ob\Z = \ZZ.
  \end{equation}
\end{defn}

\begin{defn}\label{prop:zz-piccat}
  Addition induces a permutative structure on the super integers $\Z$ as follows.
  \begin{description}
  \item[Monoidal sum] The monoidal sum on $\Z$ is given by addition of integers and the formula
    \begin{equation}\label{eq:muknuj}
      \mu_{k} + \nu_{j} = (\mu \cdot \nu)_{k+j}
    \end{equation}
    on morphisms, where $\mu \cdot \nu$ is the product in $\ZZ^\times$.
  \item[Unit and associativity] The unit object is $0$, and the unit and associativity isomorphisms are all the identity.
  \item[Symmetry] The symmetry isomorphism is given by the \term{graded sign rule} 
    \begin{equation}\label{eq:Zbetakj}
      \beta_{k,j} = \big( (-1)^{jk} \big)_{k+j} \cn k+j \to j+k.
    \end{equation}
  \end{description}
  Functoriality of addition with respect to composition in $\Z$ follows from commutativity of multiplication in $\ZZ^\times$.
  With this permutative structure, each object $k$ is part of a canonical invertible pair,
  \begin{equation}\label{eq:ul-k}
    \ul{k} = (k, -k, \id, \id).
  \end{equation}
\end{defn}

\begin{defn}\label{defn:ochi}
  Let $(A, +, 0, \beta)$ be a permutative category with an invertible pair $\ul{a} = (a,a', \eta, \epz)$. 
  Define a symmetric monoidal functor 
  \[
    \ochi_a \cn \Z \to A
  \]
  as follows.
  \begin{description}
  \item[Object assignment]
    For $k \in \ZZ$, we define
    \begin{equation}\label{eq:ochi-k}
      \ochi_a(k) = 
      \begin{cases}
        k \cdot a & \text{if $k>0$,}\\
        0 & \text{if $k=0$,} \\
        (-k) \cdot a' & \text{if $k < 0$.}
      \end{cases}
    \end{equation} 
  \item[Morphism assignment]
    The only non-identity morphisms of $\Z$ are $-1_k$ for each integer $k$.
    Using \cref{notn:8}, we define 
    \begin{equation}\label{eq:ochia-minusone}
      \ochi_a(-1_k) = \mathtt{8}_a + \id_{\ochi_a(k)}.
    \end{equation}
  \end{description}
  Functoriality of $\ochi_a$ follows from that of $+$ and the 2-Torsion of $\fet$ (\cref{lem:64=1}):
  \begin{align*}
    \ochi_a(-1_k) \circ \ochi_a(-1_k)
    & = \big( \mathtt{8}_a + \id_{\ochi_a(k)} \big) \circ \big( \mathtt{8}_a + \id_{\ochi_a(k)} \big) \\
    & = \big( \mathtt{8}_a \circ \mathtt{8}_a \big) + \big( \id_{\ochi_a(k)}  \circ  \id_{\ochi_a(k)} \big) \\
    & = \id_0 + \id_{\ochi_a(k)} \\
    & = \id_{\ochi_a(k)}\\
    & = \ochi_a( -1_k \circ -1_k).
  \end{align*}
  
  Next we equip $\ochi_a$ with the structure of a monoidal functor.
  \begin{description}
  \item[Unit constraint] Define
    \begin{equation}\label{eq:ochia-0}
      (\ochi_a)_0 = \id \cn 0 \to 0 = \ochi_a(0).
    \end{equation}
  \item[Monoidal constraint]
    We define components of $(\ochi_a)_2$ with a combination of cases and induction.
    \begin{description}
    \item[For $k,j \geq 0$]
      \begin{equation}\label{eq:ochia-2-pos}
        (\ochi_a)_{2;k,j} = \id_{(k+j)\cdot a} \cn
        \overbrace{a + \cdots + a}^{\text{$k$ copies}}
        + \overbrace{a + \cdots + a}^{\text{$j$ copies}}
        \fto{=}
        \overbrace{a + \cdots + a}^{\text{$k + j$ copies}}.
      \end{equation}
      \vspace{.25pc}
    \item[For $k,j \leq 0$]
      \begin{equation}\label{eq:ochia-2-neg}
        (\ochi_a)_{2;k,j} = \id_{(-k-j)\cdot a'} \cn
        \overbrace{a' + \cdots + a'}^{\text{$-k$ copies}}
        + \overbrace{a' + \cdots + a'}^{\text{$-j$ copies}}
        \fto{=}
        \overbrace{a' + \cdots + a'}^{\text{$-k - j$ copies}}.
      \end{equation}
      \vspace{.25pc}
    \item[For $k > 0,\;j < 0$] $(\ochi_a)_{2;k,j}$ is defined inductively as the following composite.
      \begin{equation}\label{eq:ochia-2-mix1}
        \begin{tikzpicture}[x=10ex,y=9ex,vcenter]
          \draw[0cell] 
          (0,0) node (a) {k \cdot a + (-j)\cdot a'}
          (a)++(0,-1) node (b) {(k-1)\cdot a + (-j-1)\cdot a'}
          (b)++(3.5,0) node (c) {(k+j)\cdot a}
          ;
          \draw[1cell] 
          (a) edge['] node {(k-1)\cdot a + \epz + (-j-1) \cdot a'} (b)
          (b) edge node {(\ochi_a)_{2;k-1,j+1}} (c)
          ;
        \end{tikzpicture}
      \end{equation}
      So, in this case $(\ochi_a)_{2;k,j}$ is a composite of $\min{(k, |j|)}$ morphisms of the form $\id + \epz + \id$.
      \vspace{.25pc}
    \item[For $k < 0,\;j > 0$] $(\ochi_a)_{2;k,j}$ is defined inductively as the following composite.
      \begin{equation}\label{eq:ochia-2-mix2}
        \begin{tikzpicture}[x=10ex,y=9ex,vcenter]
          \draw[0cell] 
          (0,0) node (a) {(-k) \cdot a' + j\cdot a}
          (a)++(0,-1) node (b) {(-k - 1) \cdot a' + (j-1)\cdot a}
          (b)++(3.5,0) node (c) {(k+j)\cdot a}
          ;
          \draw[1cell] 
          (a) edge['] node {(-k-1)\cdot a' + \eta^{-1} + (j-1) \cdot a} (b)
          (b) edge node {(\ochi_a)_{2;k+1,j-1}} (c)
          ;
        \end{tikzpicture}
      \end{equation}
      So, in this case, $(\ochi_a)_{2;k,j}$ is a composite of $\min{(|k|, j)}$ morphisms of the form $\id + \eta^\inv + \id$.
    \end{description}
  \end{description} 
  Naturality of the monoidal constraints with respect to the non-identity morphisms $-1_k$ follows from the definition \cref{eq:ochia-minusone} and the Eckmann-Hilton \cref{lem:EH} \cref{it:EH1}:
  the following diagram commutes because $\fet_a$ is an automorphism of the unit $0 \in A$.
  \[
    \begin{tikzpicture}[x=24ex,y=8ex]
      \draw[0cell] 
      (0,0) node (a) {\ochi_a(k)+\ochi_a(j)}
      (1,0) node (b) {\ochi_a(k+j)}
      (1,-1) node (c) {\ochi_a(k+j)}
      (0,-1) node (d) {\ochi_a(k)+ \ochi_a(j)}
      ;
      \draw[1cell] 
      (a) edge node {(\ochi_a)_{2;k,j}} (b)
      (b) edge[] node {\mathtt{8}_a + \id} (c)
      (a) edge[swap] node {\mathtt{8}_a + \id + \id} (d)
      (d) edge[swap] node {(\ochi_a)_{2;k,j}} (c)
      ;
    \end{tikzpicture}
  \]
  Naturality with respect to $-1_j$ is similar.
  The symmetric monoidal functor axioms \cref{eq:smfunctor} for $\ochi_a$ are verified in \cref{prop:ochi-mon}.
\end{defn}
\begin{rmk}\label{rmk:ochia-warn}
  Note that the definition of $\ochi_a$ depends implicitly on all the data of the invertible pair $\ul{a} = (a, a', \eta, \epz)$.
  One can show, as a special case of \cref{prop:ka-mon-trans} below, that different choices of invertible pair data with the same base object $a$ will result in constructions $\ochi$ that are isomorphic, but generally not equal, as symmetric monoidal functors.
\end{rmk}

Recall from \cref{notn:braces} that $\S$ denotes the free permutative category generated by an object $x$.
Furthermore, for an object $a$ of a permutative category $A$, \cref{notn:chi} defines $\chi_a\cn \S \to A$ as the unique strict symmetric monoidal functor with $\chi_a(x) = a$.
\begin{prop}\label{prop:ochi-mon}
  In the context of \cref{defn:ochi}, $\ochi_a$ is a symmetric monoidal functor such that the following diagram commutes.
  \begin{equation}\label{eq:ochi-triangle}
    \begin{tikzpicture}[x=14ex,y=5ex,vcenter]
      \draw[0cell] 
      (0,0) node (a) {\zS}
      (2,0) node (b) {A}
      (1,-1) node (c) {\Z}
      ;
      \draw[1cell] 
      (a) edge node {\chi_a} (b)
      (a) edge[swap] node {\chi_1} (c)
      (c) edge[swap] node {\ochi_a} (b)
      ;
    \end{tikzpicture}
  \end{equation}
\end{prop}
\begin{proof}
  The associativity axiom of \cref{eq:smfunctor} for $\ochi_a$ is commutativity of the following diagram for integers $k$, $j$, and $i$.
  \[
    \begin{tikzpicture}[x=36ex,y=8ex]
      \draw[0cell] 
      (0,0) node (a) {\ochi_a(k)+\ochi_a(j) + \ochi_a(i)}
      (1,0) node (b) {\ochi_a(k+j)+ \ochi_a(i)}
      (1,-1) node (c) {\ochi_a(k+j+i)}
      (0,-1) node (d) {\ochi_a(k)+ \ochi_a(j+i)}
      ;
      \draw[1cell] 
      (a) edge node {(\ochi_a)_{2;k,j} + \ochi_a(i)} (b)
      (b) edge[] node {(\ochi_a)_{2;k+j,i}} (c)
      (a) edge[swap] node {\ochi_a(k) + (\ochi_a)_{2;j,i}} (d)
      (d) edge[swap] node {(\ochi_a)_{2;k,j+i}} (c)
      ;
    \end{tikzpicture}
  \]
  If any of $k$, $j$, or $i$ is zero, then commutativity is immediate.
  If any consecutive two of $k, j, i$ are both positive or both negative, then one side of the above diagram is the identity, and the resulting triangle commutes by the inductive definitions \cref{eq:ochia-2-mix1,eq:ochia-2-mix2}.
  In the remaining case, where $j$ has the opposite sign of $k$ and $i$, induction reduces the verification to $(k,j,i)$ being either $(1,-1,1)$ or $(-1,1,-1)$.
  Each of those cases is equivalent, via invertibility of $\eta$, to the triangle identities \cref{eq:aa'-triang} for $\ul{a}$.
The two unit axioms of \cref{eq:smfunctor} hold trivially for $\ochi_a$.
  Thus, $\ochi_a$ is a monoidal functor.

  The symmetry axiom for $\ochi_a$ is commutativity of the following diagram for each pair of integers $k$ and $j$.
  \begin{equation}\label{eq:ochia-symm}
    \begin{tikzpicture}[x=24ex,y=8ex,vcenter]
      \draw[0cell] 
      (0,0) node (a) {\ochi_a(k)+\ochi_a(j)}
      (1,0) node (b) {\ochi_a(k+j)}
      (1,-1) node (c) {\ochi_a(j+k)}
      (0,-1) node (d) {\ochi_a(j)+ \ochi_a(k)}
      ;
      \draw[1cell] 
      (a) edge node {(\ochi_a)_{2;k,j}} (b)
      (b) edge[] node {\ochi_a(\beta)} (c)
      (a) edge[swap] node {\beta} (d)
      (d) edge[swap] node {(\ochi_a)_{2;j,k}} (c)
      ;
    \end{tikzpicture}
  \end{equation}
  Since the above arguments show that $\ochi_a$ is a monoidal functor, we reduce to checking the square above when $k,j \in \{ \pm 1 \}$.
  Note, from \cref{eq:Zbetakj}, that $\beta_{k,j} = -1_{k+j}$ in each of these four cases.
  In the two cases $k = j$, the square above commutes by \cref{it:betaaa=8aaa,it:betaa'a'=8aa'a'}.
  In the other two cases, the square above commutes by definition of $\fet_a$ \cref{notn:8} and (for $k = 1$, $j=-1$) the equality $\fet_a^\inv = \fet_a$ from the 2-Torsion of $\fet$ (\cref{lem:64=1}).
  This concludes the proof that $\ochi_a$ is a symmetric monoidal functor.

  For commutativity of the triangle \cref{eq:ochi-triangle}, note that the monoidal constraints $(\ochi_a)_{2; k,j}$ are identities for $k,j \geq 0$.
  Thus, the composite $\ochi_a \circ \chi_1$ is strict symmetric monoidal.
  Commutativity of \cref{eq:ochi-triangle} therefore follows from uniqueness of $\chi_a$ in \cref{notn:chi} and the computation
  \[
    \bigl( \ochi_a \circ \chi_1 \bigr)(x)
    = \ochi_a(1) = a = \chi_a(x).
  \]
\end{proof}

\begin{lem}\label{lem:F-ochi-pos-k}
  Let $(A, +, 0, \beta)$ be a permutative category with an invertible pair $(a,a', \eta, \epz)$.
  Suppose that $F \cn \Z \to A$ is a symmetric monoidal functor such that the  following triangle commutes.
  \begin{equation}\label{eq:F-pos-k-triangle}
    \begin{tikzpicture}[x=14ex,y=5ex,vcenter]
      \draw[0cell] 
      (0,0) node (a) {\zS}
      (2,0) node (b) {A}
      (1,-1) node (c) {\Z}
      ;
      \draw[1cell] 
      (a) edge node {\chi_a} (b)
      (a) edge[swap] node {\chi_1} (c)
      (c) edge['] node[pos=.5] (F) {F} (b)
      ;
    \end{tikzpicture}
  \end{equation}
  Then the following three statements hold.
  \begin{enumerate}
  \item $F$ agrees with $\ochi_a$ on non-negative integers:
  \begin{equation}\label{eq:Fk-kge0-lem}
    F(k) = \ochi_a(k) = k \cdot a \forspace k \ge 0.
  \end{equation}
  \item The unit and monoidal constraints of $F$ and $\ochi_a$ agree at non-negative integers:
  \begin{equation}\label{eq:F0-F2-kge0-lem}
    \begin{aligned}
      F_0 & = \id_0 = (\ochi_a)_0 \andspace\\
      F_{2; k,j} & = \id_{k\cdot a + j \cdot a} = (\ochi_a)_{2;k,j}
    \forspace k,j \ge 0.
    \end{aligned}
  \end{equation}
  \item The assignment on non-trivial morphisms is given by
  \begin{equation}\label{eq:F-on-odd-morphism-lem}
    F(-1_k) = \fet_a + \id_{F(k)} \forspace k \in \ZZ.
  \end{equation}
  \end{enumerate}
\end{lem}
\begin{proof}
  Each of the assertions in the statement depends on commutativity of \cref{eq:ochi-triangle,eq:F-pos-k-triangle}.
  For \cref{eq:Fk-kge0-lem}, observe that $\chi_1$ is surjective on the non-negative integers $k \in \Z$.
  So, for $k \ge 0$ we have
  \[
    F(k) = F(\chi_1(k \cdot x)) = \chi_a(k \cdot x)  = k \cdot a = \ochi_a(k).
  \]

  For \cref{eq:F0-F2-kge0-lem}, observe that $F\circ \chi_1 = \chi_a$ as monoidal functors, with both $\chi_1$ and $\chi_a$ strict monoidal.
  This shows that the unit constraint of $F$ is an identity, and that the monoidal constraints of $F$ at non-negative integers are also identities, agreeing with \cref{eq:ochia-0,eq:ochia-2-pos} as desired.

  We obtain \cref{eq:F-on-odd-morphism-lem} as follows.
  Using the Figure Eight \cref{notn:8} for the canonical invertible pair $\ul{1}$ from \cref{eq:ul-k}, and the symmetry \cref{eq:Zbetakj} with the monoidal sum of morphisms \cref{eq:muknuj} in $\Z$, we obtain the following for each integer $k \in \ob \Z = \ZZ$:
  \begin{equation}\label{eq:8=-1}
    \fet_1 = -1_0
    \andspace
    \fet_1 + \id_k = -1_k
    \inspace \Z.
  \end{equation}
  Then, the monoidal unity axioms \cref{eq:smfunctor} for $F$ imply that the composites along the top and bottom in the following diagram are identities.
  \[
    \begin{tikzpicture}[x=20ex,y=10ex]
      \draw[0cell] 
      (0,0) node (a) {F(k)}
      (a)++(9ex,0) node (b) {0+F(k)}
      (b)++(1.25,0) node (c) {F(0)+F(k)}
      (c)++(1,0) node (d) {F(0+k)}
      (d)++(9ex,0) node (e) {F(k)}
      (a)++(0,-1) node (a') {F(k)}
      (b)++(0,-1) node (b') {0+F(k)}
      (c)++(0,-1) node (c') {F(0)+F(k)}
      (d)++(0,-1) node (d') {F(0+k)}
      (e)++(0,-1) node (e') {F(k)}
      ;
      \draw[1cell]
      (a) edge[equal] node {} (b)
      (a') edge[equal] node {} (b')
      (d) edge[equal] node {} (e)
      (d') edge[equal] node {} (e')
      (b) edge node {F_0 + \id_{F(k)}} (c)
      (c) edge node {F_{2;0,k}} (d)
      (b') edge['] node {F_0 + \id_{F(k)}} (c')
      (c') edge['] node {F_{2;0,k}} (d')
      (b) edge['] node {\fet_a + \id_{F(k)}} (b')
      (c) edge['] node {F(\fet_1) + F(\id_{k})} (c')
      (d) edge['] node {F(\fet_1 + \id_k)} (d')
      (e) edge node {F(-1_k)} (e')
      ;
    \end{tikzpicture}\vspace{-1pc} 
  \] 
    The square at left commutes by \cref{lem:feightone-equals-eightone}, recalling that $\fet_{F(1)} = \fet_a$ depends only on the isomorphism class of $a$, by \cref{lem:8-invariant}.
  The middle square commutes by naturality of $F_2$.
  The right square commutes by the second equation in \cref{eq:8=-1}.
  Thus we have \cref{eq:F-on-odd-morphism-lem} for any integer $k$.
\end{proof}

\begin{defn}\label{defn:F-iso-ochi}
  Let $(A, +, 0, \beta)$ be a permutative category with an invertible pair $(a,a', \eta, \epz)$. 
  Suppose that $F\cn \Z \to A$ is a symmetric monoidal functor such that the inner triangle below commutes in $\permcat$.
  \begin{equation}\label{eq:F-iso-ochi-triangle}
    \begin{tikzpicture}[x=14ex,y=5ex,vcenter]
      \draw[0cell] 
      (0,0) node (a) {\zS}
      (2,0) node (b) {A}
      (1,-1) node (c) {\Z}
      ;
      \draw[1cell] 
      (a) edge node {\chi_a} (b)
      (a) edge[swap] node {\chi_1} (c)
      (c) edge node[pos=.3] (F) {F} (b)
      (c) edge[bend right=35, swap] node[pos=.7] (X) {\ochi_a} (b)
      ;
      \draw[2cell]
      node[between=F and X at .5, rotate=-60, 2label={above,\ka}] {\Rightarrow}
      ;
    \end{tikzpicture}
  \end{equation}
  Define a monoidal natural isomorphism $\ka \cn F \cong \ochi_a$ such that $\ka_1 = \id_a$ as follows.

  First, recalling $F(k) = \ochi_a(k)$ for $k \ge 0$ \cref{eq:Fk-kge0-lem} and $F_{2;k,j} = \id = (\ochi_a)_{2;k,j}$ for $k,j \ge 0$ \cref{eq:F0-F2-kge0-lem}, we define
  \begin{equation}\label{eq:kappa-k-kge0}
    \ka_k = \id_{k\cdot a} \cn F(k) \to \ochi(k) \forspace k \ge 0.
  \end{equation} 
  Second, we define $\ka_k$ for $k < 0$ inductively.
  The component at $k = -1$ is defined as the composite
  \begin{equation}\label{eq:kappa-m1}
    F(-1) \fto{\id_{F(-1)} + \epz^{-1}} F(-1) + a + a' \fto{F_{2; -1, 1} + \id_{a'}} F(0) + a' = a' = \ochi_a(-1).
  \end{equation}
  The component at $k < -1$ is defined as the composite
  \begin{equation}\label{eq:kappa-k}
    \begin{tikzpicture}[x=25ex,y=8ex]
      \draw[0cell] 
      (0,0) node (a) {F(k)}
      (a)++(0,-1) node (b) {F(k+1) + F(-1)}
      (b)++(1.3,0) node (c) {\ochi_a(k+1) + \ochi_a(-1)}
      (c)++(1.1,0) node (d) {\ochi_a(k).}
      ;
      \draw[1cell] 
      (a) edge['] node {F_{2; k+1,-1}^{-1}} (b)
      (b) edge node {\ka_{k+1} + \ka_{-1}} (c)
      (c) edge node {(\ochi_a)_{2; k+1, -1}} (d)
      ;
    \end{tikzpicture}
  \end{equation}
  Naturality of $\ka$ with respect to the non-identity morphisms $-1_k$ in $\Z$ follows from strictness of the unit $0 \in A$, functoriality of the monoidal sum $+$, and the formulas for $\ochi_a(-1_k)$ and $F(-1_k)$ in \cref{eq:ochia-minusone,eq:F-on-odd-morphism-lem}, respectively.
  The \emph{monoidal} natural isomorphism axioms \cref{eq:montransf} are verified in \cref{prop:ka-mon-trans}.
\end{defn}

\begin{prop}\label{prop:ka-mon-trans}
  In the context of \cref{defn:F-iso-ochi}, the natural isomorphism $\ka \cn F \iso \ochi_a$ is a monoidal natural isomorphism.
\end{prop}
\begin{proof}
  Since $\ka_k = \id_{k\cdot a}$ for $k \ge 0$ \cref{eq:kappa-k-kge0}, $F(0) = 0 = \ochi_a(0)$ \cref{eq:Fk-kge0-lem}, and $F_{2;k,j} = \id = (\ochi_a)_{2;k,j}$ for $k,j \ge 0$ \cref{eq:F0-F2-kge0-lem}, there is nothing to check for the unit axiom or for the monoidal axiom in the cases $k,j \ge 0$, $j=0$, or $k=0$.
  So, by symmetry, it remains to verify that the following square commutes for $j < 0$ and $k \neq 0$.
  \begin{equation}\label{eq:ka-mon-square}
    \begin{tikzpicture}[x=24ex,y=8ex,vcenter]
      \draw[0cell] 
      (0,0) node (a) {F(k) + F(j)}
      (1,0) node (b) {\ochi_a(k) + \ochi_a(j)}
      (1,-1) node (c) {\ochi_a(k+j)}
      (0,-1) node (d) {F(k+j)}
      ;
      \draw[1cell] 
      (a) edge node {\ka_k + \ka_j} (b)
      (b) edge[] node {(\ochi_a)_{2; k, j}} (c)
      (a) edge[swap] node {F_2} (d)
      (d) edge node {\ka_{k+j}} (c)
      ;
    \end{tikzpicture}
  \end{equation}

  In the special case $j = -1$, \cref{eq:ka-mon-square} commutes for $k < 0$ by definition of $\ka_{k-1}$ in \cref{eq:kappa-k}.
  If $j = -1$ and $k > 0$, \cref{eq:ka-mon-square} is the outer diagram below, by the definitions of $\ochi$ and $\ochi_2$ in \cref{eq:ochi-k,eq:ochia-2-mix1}, respectively; the computation $F(k) = k\cdot a = \ochi_a(k)$ for $k \ge 0$ in \cref{eq:Fk-kge0-lem}; and the definition $\ka_k = \id_{k\cdot a}$ for $k \ge 0$ in \cref{eq:kappa-k-kge0}.
  \begin{equation}\label{eq:ka-mon-star}
    \begin{tikzpicture}[x=15ex,y=8ex,vcenter]
      \draw[0cell] 
      (0,0) node (a) {k\cdot a + F(-1)}
      (3,0) node (b) {k\cdot a + a'}
      (3,-2) node (c) {(k-1)\cdot a}
      (0,-2) node (d) {(k-1)\cdot a}
      (a)++(1.5,-1.1) node (p) {(k\cdot a)+F(-1)+a+a'}
      ;
      \draw[1cell] 
      (a) edge node {\id_{k \cdot a} + \ka_{-1}} (b)
      (b) edge[] node {\id_{(k-1)\cdot a} + \epz} (c)
      (a) edge[swap] node {F_{2;k,-1}} (d)
      (d) edge node {\id_{(k-1)\cdot a}} (c)
      (a) edge node {\id + \id + \epz^\inv} (p)
      (p) edge node {(\star)} (b)
      ;
    \end{tikzpicture}
  \end{equation}
  In the above diagram, we define the morphism $(\star)$ as $\id_{k\cdot a} + F_{2;-1,1} + \id$, so the triangle commutes by the definition of $\ka_{-1}$ \cref{eq:kappa-m1}.
  By the monoidal associativity axiom \cref{eq:smfunctor} for $F$ with $(x,y,z) = (k,-1,+1)$ and the observation \cref{eq:F0-F2-kge0-lem} that $F_{2;k,0} = \id_{(k\cdot a)}$ for $k > 0$, the morphism $(\star)$ is also equal to the composite
  \[
  k \cdot a + F(-1) + a + a' \fto{F_{2;k, -1} + \id_a + \id_{a'}} (k-1) \cdot a + a + a' \fto{F_{2; k-1,1} + \id} k \cdot a + a'.
  \]
  Since $k>0$, we have $k-1 \geq 0$, so $F_{2; k-1,1}$ is also the identity $\id_{k \cdot a}$ by \cref{eq:F0-F2-kge0-lem}. 
  Therefore, we conclude that $(\star)$ is also equal to $F_{2;k, -1} + \id_{a+a'}$.
  Hence, the lower region of \cref{eq:ka-mon-star} commutes by functoriality of the monoidal sum and invertibility of $\epz$.
  This finishes the verification that \cref{eq:ka-mon-star} commutes for all integers $k$.

  Finally, we show commutativity of \cref{eq:ka-mon-square} for all $k$ and for $j < -1$ by induction on $j$.
  The outer diagram below is that of \cref{eq:ka-mon-square}.
  \begin{equation}\label{eq:ka-mon-square-induction}
    \begin{tikzpicture}[x=9.0ex,y=7ex,vcenter]
      \draw[0cell] 
      (0,0) node (a) {F(k) + F(j)}
      (7,0) node (b) {\ochi_a(k) + \ochi_a(j)}
      (7,-4) node (c) {\ochi_a(k+j)}
      (0,-4) node (d) {F(k+j)}
      (a)++(1.5,-1.25) node (x) {F(k)+F(j+1)+F(-1)}
      (b)++(-1.5,-1.25) node (y) {\ochi(k)+\ochi(j+1)+\ochi(-1)}
      (c)++(-1.5,1.25) node (z) {\ochi(k+j+1)+\ochi(-1)}
      (d)++(1.5,1.25) node (w) {F(k+j+1)+F(-1)}
      ;
      \draw[1cell] 
      (a) edge node {\ka_k + \ka_j} (b)
      (b) edge[swap] node {(\ochi_a)_{2; k, j}} (c)
      (a) edge node {F_2} (d)
      (d) edge node {\ka_{k+j}} (c)
      (a) edge node {\id + F_2^\inv} (x)
      (b) edge['] node {\id + \ochi_2^\inv} (y)
      (z) edge['] node {\ochi_2} (c)
      (w) edge node {F_2} (d)
      (x) edge node {\ka_k + \ka_{j+1} + \ka_{-1}} (y)
      (y) edge['] node {\ochi_2 + \id} (z)
      (x) edge node {F_2 + \id} (w)
      (w) edge node {\ka_{k+j+1} + \ka_{-1}} (z)
      ;
    \end{tikzpicture}
  \end{equation}
  In the above diagram, the two trapezoids at left and right commute by monoidal associativity \cref{eq:smfunctor} for $F$ and $\ochi$, respectively. 
  The two trapezoids at top and bottom commute because they are instances of \cref{eq:ka-mon-star}.
  The inner square commutes by induction on $j$.
  This completes the proof that $\ka$ is a monoidal natural isomorphism.
\end{proof}

Recall from \cref{prop:prepn} that the permutative category $\P$ represents the 2-functor $A \mapsto \pinv{A}$, for permutative categories $A$.
\begin{thm}\label{thm:ZZ-equiv-P1}
  Let
  \[
    K \cn \P \to \Z
  \]
  be the strict symmetric monoidal functor determined via \cref{eq:prepn} by the canonical invertible pair $\ul{1} = (1,-1,\id,\id) \in \pinv{\Z}$ \cref{eq:ul-k}.
  Then $K$ is a symmetric monoidal equivalence.
\end{thm}
\begin{proof}
  By \cref{prop:prepn}, the strict symmetric monoidal functor $K \cn \P \to \Z$ is uniquely determined by mapping the invertible pair $\ul{x} = (x,y,\eta, \epz)$ in $\P$ to $\ul{1}$.
  Using the construction in \cref{defn:ochi}, $\ul{x}$ determines a symmetric monoidal functor
  \[
    \ochi_x\cn \Z \to \P.
  \]
  We will prove that $K$ and $\ochi_x$ exhibit a symmetric monoidal equivalence between $\P$ and $\Z$ by constructing monoidal natural isomorphisms $K \circ \ochi_x \cong \id_{\Z}$ and $\id_{\P} \cong \ochi_x \circ K$. 

  First, consider the symmetric monoidal functor
  \[
    \ochi_1 \cn \Z \to \Z
  \]
  determined by \cref{defn:ochi} for the invertible pair $\ul{1}$ in $\Z$.
  Since the unit and counit for $\ul{1}$ are identities, the definitions \cref{eq:ochia-2-mix1,eq:ochia-2-mix2} show that $\ochi_1$ is equal to the identity $\id_\Z$ as a symmetric monoidal functor.
  The composite
  \[
    \Z \fto{\ochi_x} \P \fto{K} \Z
  \]
  is a symmetric monoidal functor such that the following triangles commute in $\permcat$.
  \[
    \begin{tikzpicture}[x=11ex,y=4ex,vcenter]
      \draw[0cell] 
      (0,0) node (z) {\Z}
      (z)++(45:1.4) node (p) {\P}
      (p)++(45:1.4) node (z2) {\Z}
      (z2)++(-3,0) node (s) {\S}
      ;
      \draw[1cell] 
      (s) edge['] node {\chi_1} (z)
      (z) edge['] node {\ochi_{x}} (p)
      (s) edge['] node[pos=.7] {\chi_x} (p)
      (s) edge node {\chi_1} (z2)
      (p) edge['] node {K} (z2)
      ;
    \end{tikzpicture}
  \]
  Therefore, by \cref{prop:ka-mon-trans} with $F = K \circ \ochi_x$, there is a monoidal natural isomorphism $\ka \cn K \circ \ochi_x \iso \ochi_1 = \id_\Z.$

  Second, consider the composite
  \[
    \P \fto{K} \Z \fto{\ochi_x} \P.
  \]
  Since $\ochi_x \bigl( K(x) \bigr) = \ochi_x(1) = x$, \cref{lem:pflex-id} gives a monoidal natural isomorphism $\ochi_x \circ K \iso \id_\P$, as desired.
\end{proof}

\begin{rmk}\label{rmk:yoneda-version}
  A more abstract proof of \cref{thm:ZZ-equiv-P1} can be obtained from the bicategorical Yoneda Lemma.
  For such a proof, one extends the analysis of \cref{prop:ka-mon-trans} somewhat non-trivially to show that there is an equivalence, for each permutative category $A$,
  \begin{equation}\label{eq:ZZ-rep-inv2}
    \permcat(\Z,A) \fto{\hty} \pinv{A}
  \end{equation}
  that is 2-natural in $A$.
  One would then combine \cref{eq:ZZ-rep-inv2} with the isomorphism from \cref{prop:prepn} and the equivalence from \cref{cor:pflex-equiv}:
  \begin{equation}\label{eq:its-psnat2}
    \pinv{A} \cong \permcats(\P, A)
    \xgenarrow{\simeq}{right hook->}
    \permcat(\P, A).
  \end{equation}
  To apply the bicategorical Yoneda Lemma and obtain a monoidal equivalence $\Z \hty \P$, one must show that the composite equivalence \cref{eq:its-psnat2} is pseudonatural in the variable $A$, with respect to all symmetric monoidal functors, even though the first isomorphism is \emph{a priori} only 2-natural in strict symmetric monoidal functors.

  The authors' attempts to prove the required pseudonaturality for \cref{eq:its-psnat2} use flexibility of $\P$ in a way that is comparable but somewhat more general than that of \cref{lem:pflex-id} in the direct proof of \cref{thm:ZZ-equiv-P1} above.
  Recovery of the inducing equivalence $K\cn \Z \to \P$ from the Yoneda argument, by evaluating the composite of \cref{eq:ZZ-rep-inv2,eq:its-psnat2} for the case $A = \Z$ and the identity functor $\id_\Z$, is then essentially equivalent to the direct arguments for \cref{thm:ZZ-equiv-P1} as written above.
  For these reasons, we have opted to give the direct proof of \cref{thm:ZZ-equiv-P1} and leave details of the Yoneda argument to the interested reader.
\end{rmk}

\section{Parity and coherence}
\label{sec:z-comp-2}

This section develops coherence theory for morphisms between invertible objects, beginning with the definition of parity for morphisms in $\Z$ and $\P$ (\cref{defn:parity}).
The proof of \cref{thm:free-inv-coherence} appears thereafter, following \cref{rmk:betaab-triv}.

The remainder of this section uses \cref{thm:free-inv-coherence} in a variety of examples.
\cref{example:elem} revisits some calculations from the perspective of coherence theory, while \cref{example:addl} provides some new calculations.
We end with applications to additivity of the Figure Eight morphisms in \cref{example:8ab=8a8b} and to symmetric monoidal structure of conjugation functors in \cref{example:conj-symm}.

\begin{defn}\label{defn:C2}
  Let $C_2 = \{0,1\}$ denote the additive group of order two and let $\Si C_2$ denote its \term{suspension}: the one-object groupoid with morphism group given by $C_2$.
  We will identify the elements 0 and 1 of $C_2$ by their respective parity: even and odd.
  Note that $\Si C_2$ has a unique symmetric monoidal structure in which the monoidal sum is given by addition in $C_2$ and the symmetry is the identity.
\end{defn}

\begin{defn}\label{defn:log}
  Define a monoidal functor
  \begin{equation}\label{eq:log}
    L \cn \Z \to \Si C_2
  \end{equation}
  as the unique strict monoidal functor that sends each automorphism $-1_k$ in $\Z(k,k)$ to the nontrivial morphism in $\Si C_2$.
  That is, we have
  \begin{equation}\label{eq:L-m1}
    L((-1_k)^p) = p \text{ (mod 2)},
  \end{equation}
  for each pair of integers $k$ and $p$.
  Thus, for each integer $k$, $L$ is the isomorphism between multiplicative and additive groups of order two:
  \begin{equation}\label{eq:Lkk-iso}
    L\cn \Z(k,k) = \ZZ^\times \iso C_2.
  \end{equation}
  Functoriality of $L$ holds because \cref{eq:L-m1} is a group homomorphism.
  The strict monoidal functor axioms from \cref{defn:smfunctor} hold for $L$ by \cref{eq:muknuj}.
\end{defn}

\begin{rmk}\label{rmk:L-nonsymm}
  Note that $L$ is a monoidal functor, but is not a \emph{symmetric} monoidal functor.
  The nontrivial symmetries in $\Z$, such as $\be_{1,1} = \fet_1 + 2$, are mapped to the nonidentity morphism in $\Si C_2$, but the symmetry of the unique object in $\Si C_2$ is necessarily the identity.
  Therefore $L$ does not map symmetries to symmetries, so is not a symmetric monoidal functor.
\end{rmk}

\begin{defn}[Parity for morphisms]\label{defn:parity}
  The \term{parity} of an automorphism $f \in \Z(k,k)$ is $Lf \in C_2$.
  The \term{parity} of a morphism $f$ in $\P$ is $LKf$, where $K\cn \P \fto{\hty} \Z$ is the equivalence from \cref{thm:ZZ-equiv-P1}.
  In each case, we write $\p(f) \in C_2$ to denote the parity of $f$.
\end{defn}
Since every morphism in $\Z$ is an automorphism, and $L$ is a bijection on automorphism sets \cref{eq:Lkk-iso}, two automorphisms of the same object in $\Z$ are equal if and only if they have the same parity.
Since $K\cn \P \to \Z$ is an equivalence, the morphisms of $\P$ have a similar characterization, stated as follows.
\begin{prop}\label{prop:parity}
  Two parallel morphisms $f$ and $g$ in $\P$ are equal if and only if they have the same parity.
\end{prop}

Now we collect several observations about computing parity.
Recall the notation $k \cdot x$ for $k \ge 0$ from \cref{notn:n-dot}.
\begin{rmk}[Basic observations]\label{rmk:parity-facts}\ 
  \begin{enumerate} 
  \item\label{it:parity-additivity} Since $K$ and $L$ are both monoidal functors, parity of morphisms in $\P$ and $\Z$ is additive with respect to monoidal sums and composition.

  \item\label{it:parity-beta} For $j,k \ge 0$, the parity of $\beta_{j\cdot x, k \cdot x}$ is that of $\beta_{j,k}$ in $\Z$, which is $jk$ (mod 2) by \cref{eq:Zbetakj}.

  \item\label{it:parity-beta-invs} Furthermore, recalling that $y \in \P$ is the inverse of $x$ and $K(y) = -1 \in \Z$ by \cref{thm:ZZ-equiv-P1}, the parities of $\beta_{j\cdot x, k \cdot y}$, $\beta_{j\cdot y, k\cdot x}$, and $\beta_{j\cdot y, k \cdot y}$ are all $jk$ (mod 2).

  \item\label{it:parity-8kx} For $k \ge 0$, we have $K(k \cdot x) = k \in \Z$, since $K$ is strict monoidal.
    Hence, by \cref{lem:feightone-equals-eightone}, the parity of $\fet_{k \cdot x}$ in $\P$ is that of $\fet_k$ in $\Z$, which is $k^2 \equiv k$ (mod 2) by \cref{it:betaaa=8aaa}.

  \item\label{it:parity-epz-eta}
    \Cref{thm:ZZ-equiv-P1} shows that $K\cn \P \to \Z$ is determined by the canonical invertible pair $\ul{1}$ \cref{eq:ul-k}.
    Thus, $K(\eta)= \id_{0} = K(\epz)$ and so the parities of $\eta$ and $\epz$ are even.
  \end{enumerate}
\end{rmk}

We now use the bicategorical coproducts from \cref{thm:smbperm1427} to define componentwise projections for invertible objects, in analogy with the projections of \cref{defn:ISS}.
\begin{defn}\label{defn:IPP}
  Suppose $G$ is a finite set and recall from \cref{equation:PPG-defn} that $\PP G = \coprod_{g \in G} \P\{g\}$ is the free permutative category on a finite set of invertible generators $G$.
  Then, for each $a \in G$, let $I_a$ denote the following composite of $I$ with the product projection.
  \begin{equation}\label{eq:IPP}
    \begin{tikzpicture}[x=12ex,y=8ex,baseline={(a.base)}]
      \draw[0cell] 
      (0,0) node (a0) {\PP G}
      (a0)++(8ex,0) node[text depth=3pt] (a) {\coprod_{g \in G} \P\{g\}}
      (a)++(1.3,0) node[text depth=3pt] (b) {\prod_{g \in G} \P\{g\}}
      (b)++(1,0) node (c) {\P\{a\}}
      ;
      \draw[0cell] (a0)++(3.1ex,0) node {=};
      \draw[1cell] 
      (a) edge node {I} node['] {\hty} (b)
      (b) edge node {} (c)
      ;
      \draw[1cell,rounded corners]
      (a0) -- ++(0,.5) -- node {I_a} ($(c)+(0,.5)$) -- (c)
      ;
    \end{tikzpicture}\vspace{.5pc}
  \end{equation}
  Recall from \cref{defn:parity} that parity of morphisms in $\P$ is computed via the equivalence $K \cn \P \to \Z$ from \cref{thm:ZZ-equiv-P1}.
  The \term{$a$-parity} of a morphism $f$ in $\PP G$ is the parity of the morphism $I_a f$ in $\P\{a\}$, denoted
  \begin{equation}\label{eq:apar}
    \p_a(f) = \p \bigl( I_a f \bigr) \in C_2.
  \end{equation}
\end{defn}

\begin{rmk}\label{rmk:betaab-triv}
  Suppose, in the context of \cref{defn:IPP}, that $G = \{a,b\}$ is a two-element set, and consider the symmetry isomorphism $\beta_{a,b}\cn a+b \to b + a$ in $\PP\{a,b\}$.
  Observe that each of $I_a(\beta_{a,b})$ and $I_b(\beta_{a,b})$ is an identity morphism, in $\PP\{a\}$ and $\PP\{b\}$, respectively.
  Therefore, in contrast with the odd $a$-parity of $\beta_{a,a}$ and $b$-parity of $\beta_{b,b}$ from \cref{rmk:parity-facts}~\cref{it:parity-beta}, we have
  \begin{equation}\label{eq:betaab-par-triv}
    \p_a(\beta_{a,b}) = 0 = \p_b(\beta_{a,b}).
  \end{equation}
\end{rmk}

\subsection*{Proof of Theorem~\texorpdfstring{\ref{thm:free-inv-coherence}}{1.1}}
Recall from the statement of \cref{thm:free-inv-coherence} that $s$ and $t$ are assumed to be parallel morphisms in $\PP G$, the free permutative category on a finite set of invertible generators $G$, as in \cref{defn:IPP}.
The strict symmetric monoidal functor $I$ in \cref{eq:IPP} is an equivalence in $\permcat$ by \cref{thm:smbperm1427}, so the morphisms $s$ and $t$ are equal if and only if their projections $I_a(s)$ and $I_a(t)$ are equal as morphisms of $\P\{a\}$ for each $a$ in $G$.

Note the assumption that $s$ and $t$ are parallel implies that $I_a(s)$ and $I_a(t)$ are parallel for each $a$ in $G$.
Therefore, for each $a$ in $G$, \cref{prop:parity} shows that the morphisms $I_a(s)$ and $I_a(t)$ are equal if and only if they have the same parity.
The stated assertion then follows from the definition of $a$-parity in \cref{eq:apar}. {\null\nobreak\hfill\ensuremath{\square}}

\begin{rmk}\label{rmk:parity-parallel}
  The assumption in \cref{thm:free-inv-coherence} that $s$ and $t$ are parallel is essential for just the same reasons noted in \cref{rmk:coherence-parallel} for symmetric coherence via underlying permutations.
  As discussed in \cref{rmk:betaab-triv}, the symmetry $\beta_{a,b}$ in $\PP\{a,b\}$ is an example of a morphism with even $a$- and $b$-parity that is \emph{not} an identity.
\end{rmk}

\subsection*{Examples of Theorem~\texorpdfstring{\ref{thm:free-inv-coherence}}{1.1}}

For the following examples, let $(A, +, 0, \beta)$ be a permutative category with an invertible pair $\ul{a} = (a, a', \eta, \epz)$.
The invertible pair $\ul{a}$ determines a unique strict symmetric monoidal functor $f\cn \P \to A$ by \cref{prop:prepn}.
We discuss $a$-parity for morphisms in $\P\{a\}$, resulting in statements for general $A$ via $f$.
\begin{example}[Elementary computations]\label{example:elem}
  The following are computed independently in \cref{sec:calcs}, but are consistent with \cref{thm:free-inv-coherence}.
  For each computation, one decomposes the indicated morphism to a composite of units, counits, and braidings on the objects $a$ and/or $a'$.
  Then, recall from \cref{rmk:parity-facts} that each of $\beta_{a,a}$, $\beta_{a',a'}$, $\beta_{a,a'}$, and $\beta_{a',a}$ has odd parity, while identities, units $\eta$, and counits $\epz$ each have even parity.
  \begin{enumerate}
  \item In the context of \cref{lem:even-perms}, for a permutation $\si \in \Si_n$, the $a$-parity of $\chi_a(\si)$ is the sign of $\sigma$.
  \item In the context of \cref{defn:fet-figC-figH}, each of $\fet_a$, $\figC_a$, and $\figH_a$ has odd $a$-parity. Since these are all parallel, they are equal.
  \item In the context of \cref{lem:betaaa=8aaa}, each of the morphisms in \cref{it:8a=8a',it:betaaa=8aaa,it:betaa'a'=8aa'a'} has odd $a$-parity.
  \end{enumerate}
\end{example}

\begin{example}[Additional computations]\label{example:addl}
  Each of these observations follows from \cref{thm:free-inv-coherence} by comparing $a$-parity.
  As in \cref{example:elem}, the computations of $a$-parity follow from \cref{rmk:parity-facts} and decomposing each morphism as a composite of (sums of) braidings on $a$ and/or $a'$, units, counits, and identities.
  \begin{enumerate}
  \item\label{it:addl-1} For an integer $k \ge 0$, computing $a$-parity shows that $\fet_{k\cdot a} = k \cdot \fet_a$.  Here, one uses the definition \cref{notn:8} for both sides, with the invertible pair $(k\cdot a, k\cdot a')$ having unit and counit given by suitable iterates of $\eta$ and $\epz$.
  \item\label{it:addl-2} The following two cyclic permutations have odd $a$-parity and are therefore equal
    \[
      \beta_{(a+a'+a),a'} = \beta_{a,(a'+a+a')} \cn a+a'+a+a' \to a'+a+a'+a.
    \]
  \item\label{it:addl-3} The automorphism
  \[
  \beta_{a+a', a+a'} \cn a + a' + a + a' \cong a + a' + a + a' 
  \]
  has even $a$-parity and, as an automorphism, is parallel to the identity on $a+a'+a+a'$.
  Therefore $\beta_{a+a',a+a'}$ is an identity morphism.
  \item\label{it:addl-4} In contrast with the previous example, the morphism
  \[
  \beta_{a+a, a'+a'} \cn a + a + a' + a' \cong a' + a' + a + a 
  \]
  has even $a$-parity, but is \emph{not} an identity, simply because it is not an automorphism.
  \end{enumerate}
\end{example}

\begin{rmk}\label{rmk:additional-comps}
  The conclusions in \cref{example:addl} cannot be extracted from the coherence for permutative categories in  \cref{thm:permcat-coherence-one-obj,thm:diagrcoh-s}.
  The first example \cref{it:addl-1} requires the isomorphisms $\eta, \epz$.
  For examples \cref{it:addl-2,it:addl-3}, one might replace $a'$ with a general object $b$ and consider the displayed morphisms in $\SS\{a,b\}$.
  Using this replacement for the two morphisms in \cref{it:addl-2} results in $\beta_{(a+b+a),b}$ and $\beta_{a,(b+a+b)}$, two morphisms with different underlying permutations.
  Likewise, such a replacement in \cref{it:addl-3} results in $\beta_{a+b,a+b}$, a morphism whose $a$- and $b$-permutations are each the nontrivial transposition $(1 \ 2) \in \Sigma_2$.
  One might also replace the object $a+a'$ with a general object $x$ and consider the $x$-permutation of \cref{it:addl-3}.
  Once again, doing so produces a non-identity morphism in $\SS\{x\}$.
  Furthermore, even as stated with $a'$, \cref{lem:even-perms} does not apply because the relevant morphisms are not in the image of $\chi_a\cn \S\{a\} \to \P\{a\}$.
\end{rmk}

\begin{example}[Weak invertibility]\label{example:weak-vs-strict-inv}
  The concept of parity offers an additional perspective on weak invertibility and \cref{rmk:triangle-and-8}.
  If $a$ is an invertible object of $A$, with weak inverse $a'$ and isomorphism
  \[
    \eta\cn 0 \iso a' + a,
  \]
  then \cref{lem:weak-vs-structured-inv} shows that $a$ is the base object of an invertible pair $\ul{a}$.
  Moreover, the method of proof in \cite[IV.4]{ML98Categories} or \cite[Theorem~1.4]{Gur2012Biequivalences} shows that we can take $\ul{a} = (a,a',\eta,\epz)$, with the same inverse $a'$ and unit $\eta$.

  However, the counit $\epz$ is generally not equal to the composite 
  \[
    \xi = \eta^\inv \circ \beta_{a,a'} \cn a+a' \iso 0.
  \]
  To explain this, note that the corresponding morphism $\xi$ in $\P\{a\}$ will generally have the \emph{opposite} $a$-parity from $\eta$.
  Hence, the triangle identities \cref{eq:aa'-triang} do \emph{not} necessarily hold in $A$ for $(a,a',\eta,\xi)$.
  Notably, they do not hold when $A = \P\{a\}$.
\end{example}

\begin{example}[Additivity of $\fet$]\label{example:8ab=8a8b}
  If $\ul{b} = (b,b',\de,\ga)$ is another invertible pair in $A$, then $a+b$ is invertible with inverse $b' + a'$.
  In this example we use $a$-parity and $b$-parity in $\PP\{a,b\} = \P\{a\} \coprod \P\{b\}$ to determine an invertible pair $\ul{a+b}$ and show
  \begin{equation}\label{eq:8ab=8a8b}
    \fet_{a+b} = \fet_a + \fet_b
  \end{equation}
  as automorphisms of $0$.
  This equality holds first in $\PP\{a,b\}$, and therefore also in any permutative category $A$.
  Although this result is well-known in the trace literature (e.g., \cite[Section~5]{PS2014Traces}), our approach is somewhat novel and demonstrates how to apply \cref{thm:free-inv-coherence} more generally.
  
  To define the unit and counit for $\ul{a+b}$, consider the solid arrow composites around the following diagrams in $\PP\{a,b\}$.
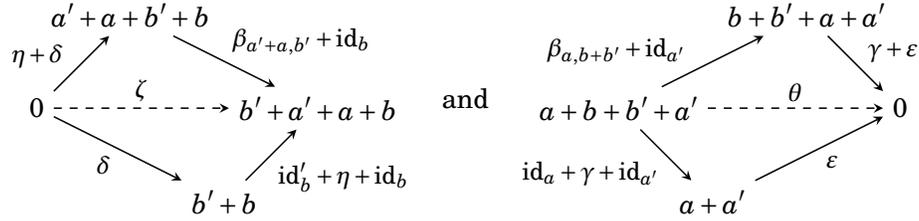
  \begin{equation}\label{eq:ze-th}
    \begin{tikzpicture}[x=15ex,y=15ex,vcenter]
      \draw[0cell] 
      (0,0) node (a) {0}
      (a)++(.5,.5) node (b) {a' + a + b' + b}
      (b)++(1,-.5) node (c) {b' + a' + a + b}
      (a)++(1,-.5) node (d) {b' + b}
      ;
      \draw[1cell] 
      (a) edge node[pos=.3] {\eta + \de} (b)
      (b) edge node {\beta_{a' + a,b'} + \id_b} (c)
      (a) edge['] node {\de} (d)
      (d) edge['] node {\id_b' + \eta + \id_b} (c)
      (a) edge[dashed] node {\ze} (c)
      ;
    \end{tikzpicture}
    \andspace
    \begin{tikzpicture}[x=15ex,y=-15ex,vcenter]
      \draw[0cell] 
      (0,0) node (a) {a + b + b' + a'}
      (a)++(.5,.5) node (b) {a + a'}
      (b)++(1,-.5) node (c) {0}
      (a)++(1,-.5) node (d) {b + b' + a + a'}
      ;
      \draw[1cell] 
      (a) edge['] node {\id_a + \ga + \id_{a'}} (b)
      (b) edge['] node {\epz} (c)
      (a) edge node[pos=.3] {\beta_{a,b+b'} + \id_{a'}} (d)
      (d) edge node[pos=.6] {\ga + \epz} (c)
      (a) edge[dashed] node {\theta} (c)
      ;
    \end{tikzpicture}
  \end{equation} 
  In the above diagrams, the $a$-parity of each composite is computed using \cref{rmk:parity-facts,eq:betaab-par-triv}.
  For example, we have the following for the upper composite in the diagram at left:
  \[
    \p_a\bigl( (\beta_{a'+a,b'} + \id_b) \circ (\eta + \de) \bigr)
    = \p_a(\beta_{a'+a,b'} + \id_b) + \p_a(\eta + \de)
    = 0.
  \]
  Likewise, each of the other three composites has even $a$-parity, and all four composites have even $b$-parity.
  This shows that the outer diagrams commute.
  We define $\ze$ and $\theta$ as the respective dashed morphisms at left and right.
  The triangle identities \cref{eq:aa'-triang} for $\ze$ and $\theta$ follow by considering the lower composites in \cref{eq:ze-th}.
  Therefore, $\ul{a+b} = (a+b, b' + a', \ze, \theta)$ is an invertible pair.

  To verify \cref{eq:8ab=8a8b}, consider the following diagram.
  Recalling \cref{notn:8}, the composites along the top and bottom are $\fet_{a+b}$ and $\fet_a + \fet_b$, respectively.
  
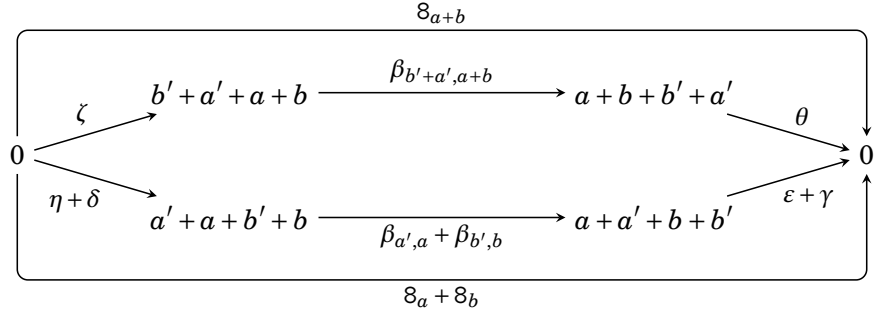
\begin{equation}\label{eq:8ab=8a8b-diagram}
    \begin{tikzpicture}[x=17ex,y=10ex,vcenter]
      \draw[0cell] 
      (0,0) node (a) {0}
      (a)++(1,.5) node (b) {b' + a' + a + b}
      (b)++(2,0) node (c) {a + b + b' + a'}
      (c)++(1,-.5) node (d) {0}
      (a)++(1,-.5) node (b') {a' + a + b' + b}
      (b')++(2,0) node (c') {a + a' + b + b'}
      ;
      \draw[1cell] 
      (a) edge node {\ze} (b)
      (b) edge node {\beta_{b'+a',a+b}} (c)
      (c) edge node {\theta} (d)
      (a) edge[',pos=.6] node {\eta + \de} (b')
      (b') edge['] node {\beta_{a',a}+\beta_{b',b}} (c')
      (c') edge[',pos=.4] node {\epz + \ga} (d)
      ;
      \draw[1cell,rounded corners]
      (a) -- ++(0,1) -- node {\fet_{a+b}} ($(d)+(0,1)$) -- (d)
      ;
      \draw[1cell,rounded corners]
      (a) -- ++(0,-1) -- node['] {\fet_a + \fet_{b}} ($(d)+(0,-1)$) -- (d)
      ;
    \end{tikzpicture}
  \end{equation}
  The $a$- and $b$-parities of the respective composites are then computed via \cref{rmk:parity-facts,eq:betaab-par-triv}, so that we have
  \[
    \p_a(\fet_{a+b})
    = 1 = \p_a(\fet_a) \andspace
    \p_b(\fet_{a+b}) 
    = 1 = \p_b(\fet_b). 
  \]
  Since $\p_b(\fet_a) = 0 = \p_a(\fet_b)$, the additivity of parity in \cref{rmk:parity-facts}~\cref{it:parity-additivity} confirms \cref{eq:8ab=8a8b} as desired.
\end{example}

Now we give a slightly more general use of the coherence results above.
Since \cref{thm:smbperm1427} applies to arbitrary permutative categories, one could use it to determine equality of morphisms that involve some general and some invertible objects in a permutative category.
Recalling \cref{eq:SSG-coprodS,equation:PPG-defn}, we have an equivalence as follows for finite sets $G$ and $H$:
\begin{equation}\label{eq:SS-and-PP}
  (\SS G) \coprod (\PP H) \hty \left(\prod_{g \in G} \S\{g\}\right) \times \left(\prod_{h \in H} \P\{h\}\right). 
\end{equation}

\begin{example}[Conjugation]\label{example:conj-symm}
  Suppose that $A$ is a permutative category with an invertible pair $\ul{a} = (a,a',\eta,\epz)$.
  There is a functor $(?)^a$ called \term{conjugation by $a$} and given on objects $z$ and morphisms $f$ of $A$ by 
  \[
    z \mapsto (z)^a = a'+ z + a \andspace f \mapsto (f)^a = \id_{a'} + f + \id_{a}.
  \]
  Furthermore, there is a monoidal constraint given by the counit of $\ul{a}$:
  \[
    z^a +  w^a = a' + z + a + a' + w + a \fto{id + \epz + \id} a' + z + w + a = (z + w )^a.
  \]
  Since the unit of $A$ is strict, we take $\eta$ for the unit constraint of $(?)^a\cn A \to A$.
  One can use \cref{eq:SS-and-PP} to verify the symmetric monoidal functor axioms \cref{eq:smfunctor} for conjugation by $a$ in certain coproducts of $\S$ and $\P$.
  For example, the symmetry axiom at $z,w \in A$ follows from commutativity of the following diagram in $\S\{z\} \coprod \S\{w\} \coprod \P\{a\}$.
  \[
    \begin{tikzpicture}[x=12ex,y=9ex]
      \draw[0cell] 
      (0,0) node (a) {(a' + z + a) + (a' + w + a)}
      (a)++(3,0) node (b) {a' + (z + w) + a}
      (b)++(0,-1) node (c) {a' + (w + z) + a}
      (a)++(0,-1) node (d) {(a' + w + a) + (a' + z + a)}
      ;
      \draw[1cell] 
      (a) edge node {\id + \epz + \id} (b)
      (d) edge node {\id + \epz + \id} (c)
      (a) edge['] node {\beta_{a' + z + a\,,\,a' + w + a}} (d)
      (b) edge node {\id + \beta_{z,w} + \id} (c)
      ;
    \end{tikzpicture}
  \]
  Using \cref{eq:SS-and-PP} and projecting to components, one verifies that the above diagram commutes because (\textit{i}) the $z$- and $w$-permutations of each morphism are identities and (\textit{ii}) the $a$-parity of each morphism is even.
  Therefore, the composites around the diagram are equal in $\S\{z\} \times \S\{w\} \times \P\{a\}$, and hence also in $\S\{z\} \coprod \S\{w\} \coprod \P\{a\}$ by \cref{thm:smbperm1427}.
  The corresponding diagram in $A$, being the image of the above formal diagram, therefore also commutes.

  Next, one can observe that the isomorphisms
  \begin{equation}\label{eq:conja-id}
    a' + z + a \fto{\id_{a'} + \beta_{z,a}} a' + a + z \fto{\eta^\inv + \id_z} z
    \foreachspace z \in A
  \end{equation}
  give a natural isomorphism $(?)^a \to \Id_A$.
  The monoidal transformation axioms \cref{eq:montransf} for this natural isomorphism follow from commutativity of the following diagram in $\S\{z\} \coprod \S\{w\} \coprod \P\{a\}$.
  \[
    \begin{tikzpicture}[x=12ex,y=9ex]
      \draw[0cell] 
      (0,0) node (a) {(a' + z + a) + (a' + w + a)}
      (a)++(3,0) node (b) {a' + (z + w) + a}
      (b)++(0,-1) node (c) {(a' + a) + z + w}
      (c)++(0,-1) node (d) {z + w}
      (a)++(0,-1) node (e) {(a' + a) + z + (a' + a) + w}
      (e)++(0,-1) node (f) {z+w}
      ;
      \draw[1cell] 
      (a) edge node {\id + \epz + \id} (b)
      (b) edge node {\id + \beta_{z+w,a}} (c)
      (c) edge node {\eta^\inv + \id} (d)
      (a) edge['] node {\id + \beta_{z,a} + \id + \beta_{w,a}} (e)
      (e) edge['] node {\eta^\inv + \id + \eta^\inv + \id} (f)
      (f) edge node {\id} (d)
      ;
    \end{tikzpicture}
  \]
  To verify that the above diagram commutes, one can again use \cref{eq:SS-and-PP}, project to components, and observe that each morphism has trivial $z$-permutation, trivial $w$-permutation, and even $a$-parity.
  As above, this implies that the corresponding diagram in $A$ also commutes.

  So, conjugation by an invertible object $a$ is a symmetric monoidal functor $A \to A$ that is monoidal naturally isomorphic to the identity on $A$.
  Likewise, conjugation by $a'$, the base object of the invertible pair $\ul{a'} = (a',a,\epz^\inv,\eta^\inv)$ from \cref{lem:inv-switch}~\cref{it:inv-switch-pair}, is monoidal naturally isomorphic to the identity on $A$.

  There is also a natural isomorphism $(?)^a \to (?)^{a'}$ with components
  \begin{equation}\label{eq:conja-not-conja'}
    a' + z + a \fto{\id + \beta_{z,a}} a' + a + z \fto{\beta_{a',a+z}} a + z + a'
    \foreachspace z \in A.
  \end{equation}
  However, this is generally \emph{not} a monoidal natural isomorphism!
  In particular, the following monoidal naturality diagram does \emph{not} commute in $\S\{z\} \coprod \S\{w\} \coprod \P\{a\}$, because the $a$-parity of the left-bottom composite is even, while the $a$-parity of the top-right composite is odd.
  \[
    \begin{tikzpicture}[x=12ex,y=9ex]
      \draw[0cell] 
      (0,0) node (a) {(a' + z + a) + (a' + w + a)}
      (a)++(3,0) node (b) {a' + (z + w) + a}
      (b)++(0,-1) node (c) {a' + a + (z + w)}
      (c)++(0,-1) node (d) {a + (z + w) + a'}
      (a)++(0,-1) node (e) {(a' + a + z) + (a' + a + w)}
      (e)++(0,-1) node (f) {(a + z + a') + (a + w + a')}
      ;
      \draw[1cell] 
      (a) edge node {\id + \epz + \id} (b)
      (b) edge node {\id + \beta_{z+w,a}} (c)
      (c) edge node {\beta_{a',a + z + w}} (d)
      (a) edge['] node {\id + \beta_{z,a'} + \id + \beta_{w,a'}} (e)
      (e) edge['] node {\beta_{a,a' + z} + \beta_{a,a' + w}} (f)
      (f) edge node {\id + \eta^\inv + \id} (d)
      ;
    \end{tikzpicture}
  \]

  To conclude, conjugation by $a$ and conjugation by $a'$ are each monoidal naturally isomorphic to the identity on $A$, and hence to each other, but the components \cref{eq:conja-not-conja'} do not define that isomorphism.
  Instead, one must use \cref{eq:conja-id} and its corresponding inverse for $(?)^{a'}$.
  The component of this composite at the unit object $z = 0$ is the composite
  \begin{equation}\label{eq:notbetaa'a}
    a' + a \fto{\eta^\inv} 0 \fto{\epz^\inv} a + a'.
  \end{equation}
  This composite is the \emph{even-parity} morphism in the two-element hom set $\P\{a\}\bigl( a'+a,a+a' \bigr) \iso \Z \bigl( 0,0 \bigr) = \ZZ^\times$.
  Recalling \cref{rmk:parity-facts}~\cref{it:parity-beta-invs}, the odd-parity morphism of the same set is $\beta_{a',a}$.
\end{example}

\section*{Declarations}
The authors have no competing interests to declare that are relevant to the content of this article.
No funds, grants, or other support was received.

\noindent Corresponding author: Nick Gurski



\begin{thebibliography}{BKP89}

\bibitem[Abr10]{Abramsky2009No}
S.~Abramsky, \emph{No-cloning in categorical quantum mechanics}, Semantic
  Techniques in Quantum Computation, Cambridge University Press, 2010,
  pp.~1--28. \doi{10.1017/CBO9781139193313.002}

\bibitem[BL04]{BL2004Higher}
J.~C. Baez and A.~D. Lauda, \emph{Higher-dimensional algebra. {V}. 2-groups},
  Theory Appl. Categ. \textbf{12} (2004), no.~14, 423--491.

\bibitem[BKP89]{BKP1989Two}
R.~{Blackwell}, G.~{Kelly}, and A.~{Power}, \emph{Two-dimensional monad
  theory}, {J. Pure Appl. Algebra} \textbf{59} (1989), no.~1, 1--41.
  \doi{10.1016/0022-4049(89)90160-6}

\bibitem[Bra20]{Bra2020Braided}
O.~Braunling, \emph{Braided categorical groups and strictifying associators},
  Homology Homotopy Appl. \textbf{22} (2020), no.~2, 295--321 (English).
  \doi{10.4310/HHA.2020.v22.n2.a19}

\bibitem[CG]{CG25Operads}
A.~Corner and N.~Gurski, \emph{Operads and equivariance}. \arxiv{2603.19854v1}

\bibitem[Dug14]{Dug2014Coherence}
D.~Dugger, \emph{Coherence for invertible objects and multigraded homotopy
  rings}, Algebr. Geom. Topol. \textbf{14} (2014), no.~2, 1055--1106 (English).
  \doi{10.2140/agt.2014.14.1055}

\bibitem[GK14]{GK2014Symmetric}
N.~Ganter and M.~Kapranov, \emph{Symmetric and exterior powers of categories},
  Transform. Groups \textbf{19} (2014), no.~1, 57--103 (English).
  \doi{10.1007/s00031-014-9255-z}

\bibitem[Gur12]{Gur2012Biequivalences}
N.~Gurski, \emph{Biequivalences in tricategories}, Theory Appl. Categ.
  \textbf{26} (2012), no.~14, 349--384.

\bibitem[GJ25]{GJalmorcoh}
N.~Gurski and N.~Johnson, \emph{Universal pseudomorphisms, with applications to
  diagrammatic coherence for braided and symmetric monoidal functors},
  Compositionality \textbf{7} (2025), no.~3, 1--68.
  \doi{10.46298/compositionality-7-3}

\bibitem[GJO19]{GJO2019Topological}
N.~Gurski, N.~Johnson, and A.~M. Osorno, \emph{Topological invariants from
  higher categories}, Notices Amer. Math. Soc. \textbf{66} (2019), no.~8,
  1225--1237. \doi{10.1090/noti1934}

\bibitem[GJO24]{GJOsmbperm}
N.~Gurski, N.~Johnson, and A.~M. Osorno, \emph{The symmetric monoidal
  2-category of permutative categories}, High. Struct. \textbf{8} (2024),
  no.~1, 244--320 (English). \doi{10.21136/HS.2024.06}

\bibitem[HZ]{HZ2023Duality}
S.~Halbig and T.~Zorman, \emph{Duality in monoidal categories}.
  \doi{10.48550/arXiv.2301.03545} \arxiv{2301.03545v3}

\bibitem[HV19]{HV2019Categories}
C.~Heunen and J.~Vicary, \emph{{C}ategories for {Q}uantum {T}heory: {A}n
  {I}ntroduction}, Oxford Graduate Texts in Mathematics, Oxford University
  Press, 2019. \doi{10.1093/oso/9780198739623.001.0001}

\bibitem[Hir03]{Hir03Model}
P.~S. Hirschhorn, \emph{Model categories and their localizations}, Math. Surv.
  Monogr., vol.~99, American Mathematical Society, Providence, RI, 2003.

\bibitem[Hor20]{Hor2020Cohomology}
P.~Horst, \emph{Cohomology of {P}icard {C}ategories}, Ph.D. thesis, Ohio State
  University, 2020, Available at
  \url{https://etd.ohiolink.edu/acprod/odb_etd/etd/r/1501/10?clear=10&p10_accession_num=osu1587396997809887}.

\bibitem[Hov99]{hovey1999mc}
M.~Hovey, \emph{Model {C}ategories}, Math. Surv. Monogr., vol.~63, Providence,
  RI: American Mathematical Society, 1999 (English).

\bibitem[JO12]{JO2012Modeling}
N.~Johnson and A.~M. Osorno, \emph{Modeling stable one-types}, Theory Appl.
  Categ. \textbf{26} (2012), no.~20, 520--537.

\bibitem[JY21]{JY212Dim}
N.~Johnson and D.~Yau, \emph{{2}-{D}imensional {C}ategories}, Oxford University
  Press, New York, 2021. \doi{10.1093/oso/9780198871378.001.0001}
  \arxiv{2002.06055}

\bibitem[JSV96]{JSV1996Traced}
A.~Joyal, R.~Street, and D.~Verity, \emph{Traced monoidal categories}, Math.
  Proc. Camb. Philos. Soc. \textbf{119} (1996), no.~3, 447--468 (English).
  \doi{10.1017/S0305004100074338}

\bibitem[JT91]{JT1991Strong}
A.~Joyal and M.~Tierney, \emph{Strong stacks and classifying spaces}, Category
  theory, {Proc}. {Int}. {Conf}., {Como}/{Italy} 1990, Lecture Notes in
  Mathematics, vol. 1488, Springer Berlin, Heidelberg, 1991, pp.~213--236
  (English). \doi{10.1007/BFb0084222}

\bibitem[Kap21]{Kap2021Supergeometry}
M.~Kapranov, \emph{Supergeometry in mathematics and physics}, New Spaces in
  Physics: Volume 2: Formal and Conceptual Reflections (2021), 114.

\bibitem[Kel72a]{Kelly1972Coherence}
G.~M. Kelly, \emph{An abstract approach to coherence}, Coherence in
  {Categories}, {Lect}. {Notes} {Math}. 281, 106-147 (1972)., 1972.
  \doi{10.1007/bfb0059557}

\bibitem[Kel72b]{Kelly1972Many}
\bysame, \emph{Many-variable functorial calculus. {I}}, Coherence in
  {Categories}, {Lect}. {Notes} {Math}. 281, 66-105 (1972)., 1972.
  \doi{10.1007/bfb0059556}

\bibitem[Kel74a]{Kel1974Coherence}
\bysame, \emph{Coherence theorems for lax algebras and for distributive laws},
  Category {S}eminar ({P}roc. {S}em., {S}ydney, 1972/1973), Lecture Notes in
  Math., Vol. 420, 1974, pp.~281--375.

\bibitem[Kel74b]{Kelly1974Doctrinal}
\bysame, \emph{Doctrinal adjunction}, Category {S}eminar ({P}roc. {S}em.,
  {S}ydney, 1972/1973), Lecture Notes in Mathematics, vol. 420, Springer,
  Berlin, 1974, pp.~257--280 (English). \doi{10.1007/BFb0063105}

\bibitem[Kel89]{Kel89Elementary}
\bysame, \emph{Elementary observations on {$2$}-categorical limits}, Bull.
  Austral. Math. Soc. \textbf{39} (1989), no.~2, 301--317.
  \doi{10.1017/S0004972700002781}

\bibitem[KL80]{KL1980Coherence}
G.~M. Kelly and M.~L. Laplaza, \emph{Coherence for compact closed categories},
  J. Pure Appl. Algebra \textbf{19} (1980), 193--213 (English).
  \doi{10.1016/0022-4049(80)90101-2}

\bibitem[Lac07]{Lack2007Homotopy}
S.~Lack, \emph{Homotopy-theoretic aspects of 2-monads}, J. Homotopy Relat.
  Struct. \textbf{2} (2007), no.~2, 229--260 (English).

\bibitem[Lac02a]{Lac02Codescent}
S.~Lack, \emph{Codescent objects and coherence}, J. Pure Appl. Algebra
  \textbf{175} (2002), no.~1-3, 223--241, Special volume celebrating the 70th
  birthday of Professor Max Kelly. \doi{10.1016/S0022-4049(02)00136-6}

\bibitem[Lac02b]{Lac02Quillen2}
\bysame, \emph{A {Q}uillen model structure for 2-categories}, $K$-Theory
  \textbf{26} (2002), no.~2, 171--205. \doi{10.1023/A:1020305604826}

\bibitem[Lac10]{Lac10Companion}
\bysame, \emph{A 2-categories companion}, Towards higher categories, IMA Vol.
  Math. Appl., vol. 152, Springer, New York, 2010, pp.~105--191.
  \doi{10.1007/978-1-4419-1524-5\_4} \arxiv{math/0702535}

\bibitem[Lap83]{Lap1983Coherence}
M.~L. Laplaza, \emph{Coherence for categories with group structure: {An}
  alternative approach}, J. Algebra \textbf{84} (1983), 305--323 (English).
  \doi{10.1016/0021-8693(83)90081-9}

\bibitem[ML63]{MLan63Natural}
S.~Mac~Lane, \emph{Natural associativity and commutativity}, Rice Univ. Studies
  \textbf{49} (1963), no.~4, 28--46.

\bibitem[ML98]{ML98Categories}
\bysame, \emph{Categories for the working mathematician}, 2nd ed., Grad. Texts
  Math., vol.~5, New York, NY: Springer, 1998 (English).
  \doi{10.1007/978-1-4757-4721-8}

\bibitem[May72]{may72geo}
J.~P. May, \emph{The {G}eometry of {I}terated {L}oop {S}paces}, Lectures Notes
  in Mathematics, vol. 271, Springer-Verlag, Berlin, 1972.

\bibitem[May74]{May1974Einfty}
\bysame, \emph{{$E_{\infty}$} spaces, group completions, and permutative
  categories}, New developments in topology ({P}roc. {S}ympos. {A}lgebraic
  {T}opology, {O}xford, 1972), London Math. Soc. Lecture Note Ser, vol.~11,
  Cambridge Univ. Press, London, 1974, pp.~61--93.
  \doi{10.1017/CBO9780511662607.008}

\bibitem[PS14]{PS2014Traces}
K.~Ponto and M.~Shulman, \emph{Traces in symmetric monoidal categories}, Expo.
  Math. \textbf{32} (2014), no.~3, 248--273 (English).
  \doi{10.1016/j.exmath.2013.12.003}

\bibitem[Rie16]{Rie2017CTC}
E.~Riehl, \emph{Category {T}heory in {C}ontext}, Aurora: Dover Modern Math
  Originals, Dover Publications, 2016.

\bibitem[Ulb84]{Ulb1984Kohaerenz}
K.-H. Ulbrich, \emph{Koh{\"a}renz in {Kategorien} mit {Gruppenstruktur}.
  {III}}, J. Algebra \textbf{88} (1984), 292--316 (German).
  \doi{10.1016/0021-8693(84)90102-9}

\end{thebibliography}

\providecommand{\bysame}{\leavevmode\hbox to3em{\hrulefill}\thinspace}
\providecommand{\MR}{\relax\ifhmode\unskip\space\fi MR }
\providecommand{\MRhref}[2]{%
  \href{http://www.ams.org/mathscinet-getitem?mr=#1}{#2}
}
\providecommand{\doi}[1]{%
  doi:\href{https://dx.doi.org/#1}{\nolinkurl{#1}}}
\providecommand{\arxiv}[1]{%
  arXiv:\href{https://arxiv.org/abs/#1}{#1}}

\end{document}